\documentclass[reqno,12pt]{amsart}
\usepackage{amsfonts}
\usepackage{bbm}
\usepackage{}
\numberwithin{equation}{section}
\usepackage{array}

\usepackage{indentfirst}
\usepackage{color}
\usepackage{amssymb}
\usepackage{mathrsfs}
\usepackage{xy}
\xyoption{all}
\def\Ext{\mbox{\rm Ext}\,} \def\Hom{\mbox{\rm Hom}} \def\dim{\mbox{\rm dim}\,} \def\Iso{\mbox{\rm Iso}}
\def\lr#1{\langle #1\rangle}    
   \def\im{\mbox{\rm Im}\,}

\def\Aut{\mbox{\rm Aut}\,}\def\A{\mathcal{A}\,} 
\def\cone{\mbox{\rm cone}\,}

\theoremstyle{plain}
\newtheorem{theorem}{\bf Theorem}[section]
\newtheorem{lemma}[theorem]{\bf Lemma}
\newtheorem{corollary}[theorem]{\bf Corollary}
\newtheorem{proposition}[theorem]{\bf Proposition}

\theoremstyle{definition}
\newtheorem{definition}[theorem]{\bf Definition}
\newtheorem{remark}[theorem]{\bf Remark}
\newtheorem{example}[theorem]{\bf Example}

\newcommand{\bt}{\begin{theorem}}
\newcommand{\et}{\end{theorem}}
\newcommand{\bl}{\begin{lemma}}
\newcommand{\el}{\end{lemma}}
\newcommand{\bd}{\begin{definition}}
\newcommand{\ed}{\end{definition}}
\newcommand{\bc}{\begin{corollary}}
\newcommand{\ec}{\end{corollary}}
\newcommand{\bp}{\begin{proof}}
\newcommand{\ep}{\end{proof}}
\newcommand{\bx}{\begin{example}}
\newcommand{\ex}{\end{example}}
\newcommand{\br}{\begin{remark}}
\newcommand{\er}{\end{remark}}
\newcommand{\be}{\begin{equation}}
\newcommand{\ee}{\end{equation}}
\newcommand{\ba}{\begin{align}}
\newcommand{\ea}{\end{align}}
\newcommand{\bn}{\begin{enumerate}}
\newcommand{\en}{\end{enumerate}}
\newcommand{\bcs}{\begin{cases}}
\newcommand{\ecs}{\end{cases}}

\makeatletter
\renewcommand{\section}{\@startsection{section}{1}{0mm}
  {-\baselineskip}{0.5\baselineskip}{\bf\leftline}}
\makeatother

\begin{document}

\title[Hall algebras associated to root categories]{Hall algebras associated to root categories}
\author{Haicheng Zhang}
\address{Institute of Mathematics, School of Mathematical Sciences, Nanjing Normal University,
 Nanjing 210023, P. R. China.\endgraf}
\email{zhanghc@njnu.edu.cn}

\begin{abstract}
Let $\mathcal {A}$ be a finitary hereditary abelian category. We define a Hall algebra for the root category of $\mathcal {A}$ by applying the derived Hall numbers of the bounded derived category $D^b(\mathcal {A})$, which is proved to be isomorphic to the Drinfeld double Hall algebra of $\mathcal {A}$. In the appendix, we also define the 1-periodic derived Hall algebra via the derived Hall numbers of $D^b(\mathcal {A})$.
\end{abstract}

\maketitle

\section{Introduction}

The Hall algebra of a finite dimensional algebra over a finite field was introduced by Ringel \cite{R90} in 1990. Ringel \cite{R90a} proved that the Ringel--Hall algebra of a representation finite hereditary algebra is isomorphic to the positive part of the corresponding quantum enveloping algebra. Green \cite{Gr95} gave a bialgebra structure on the Ringel--Hall algebra $\mathfrak{H}_v(A)$ of any hereditary algebra $A$ and showed that the composition subalgebra of $\mathfrak{H}_{v}(A)$ generated by simple $A$-modules provides a realization of the positive part of the corresponding quantum enveloping algebra. Afterwards, Xiao \cite{Xiao} obtained the antipode for $\mathfrak{H}_v(A)$, and provided a realization of the whole quantum enveloping algebra by constructing the Drinfeld double of the extended Ringel--Hall algebra of $A$.

At the very start, in order to give an intrinsic realization of the entire quantum enveloping algebra via Hall algebra approach, one tried to define Hall algebras for triangulated categories. In 2006, To\"en \cite{Toen} defined a Hall algebra, called derived Hall algebra, for a differential graded category satisfying some finiteness conditions.
Later on, Xiao and Xu \cite{XiaoXu} generalised To\"en's construction to any triangulated category satisfying certain homological finiteness conditions, which can be the bounded derived category of a finitary hereditary abelian category $\A$. However, the root category of $\A$ does not satisfy the homological finiteness conditions (cf. \cite{XiaoXu}). In other words, the Hall algebras of root categories have not been defined. In fact, more generally, the Hall algebras of odd periodic triangulated categories with some finiteness conditions have been defined in \cite{XuChen}. Nevertheless, the Hall algebras of even periodic triangulated categories still have no definitions. Some experts expect to give a realization of the whole quantum enveloping algebra via a certain Hall algebra of a root category, for example, see the introductions of \cite{Kapranov,Gor2}.

Instead of the Hall algebras of root categories, Bridgeland \cite{Br} considered the Ringel--Hall algebra of the category of $2$-periodic complexes of projective modules over a hereditary algebra $A$. By taking some localizations and reductions with respect to contractible complexes, he obtained a realization of the entire quantum enveloping algebra via such localized Hall algebra, called Bridgeland's Hall algebra of $A$, which was proved by Yanagida \cite{Yan} to be isomorphic to the Drinfeld double Hall algebra of $A$. This provides a beautiful and intrinsic realization of the entire quantum enveloping algebra via Hall algebra approach.
Inspired by the work of Bridgeland, Chen and Deng \cite{ChenD} introduced Bridgeland's
Hall algebras of $m$-periodic complexes of hereditary algebras for any positive integer $m$.

In order to generalise Bridgeland's construction to any hereditary abelian category which may not have enough projectives, inspired by the work of Gorsky \cite{Gorsky}, Lu and Peng \cite{LP} introduced the semi-derived Ringel--Hall algebra $\mathcal {S}\mathcal {D}\mathcal {H}_{\mathbb{Z}_2}(\A)$ of any hereditary abelian category $\A$, which is the localization of a certain quotient algebra of the Ringel--Hall algebra of the category of $2$-periodic complexes over $\A$. They also proved that the achieved algebra $\mathcal {S}\mathcal {D}\mathcal {H}_{\mathbb{Z}_2}(\A)$ is isomorphic to the Drinfeld double Hall algebra of $\A$.

As mentioned above, the derived Hall algebra of the bounded derived category $D^b(\A)$ of a finitary hereditary abelian category $\A$ has been defined. In this paper, by analysing the relations between the extensions in $D^b(\A)$ and the root category $\mathcal {R}(\A)$ of $\A$, we define a Hall algebra $\mathcal {D}\mathcal {H}_2(\A)$ for the root category $\mathcal {R}(\A)$
by applying the (dual) derived Hall numbers of $D^b(\A)$. In fact, in order to get a realization of the whole quantum enveloping algebra via the Hall algebra of the root category, we have appended the $K$-elements in the basis elements of the Hall algebra $\mathcal {D}\mathcal {H}_2(\A)$. The main results of this paper are to prove that $\mathcal {D}\mathcal {H}_2(\A)$ is an associative algebra and is isomorphic to the Drinfeld double Hall algebra of $\A$. The main tools used in proving these main results are Green's formula on Ringel--Hall numbers and the associativity formula of Hall algebras.
Compared with Bridgeland's Hall algebra and the semi-derived Ringel--Hall algebra, the algebra $\mathcal {D}\mathcal {H}_2(\A)$ does not involve quotients or localizations of algebras, and its basis is straightforward and structure constants are counting triangles in some sense. Even so, we still think that the constructions of Bridgeland's Hall algebra and the semi-derived Ringel-Hall algebra are innovative and useful.

The paper is organized as follows: we recall the definitions and properties of Ringel--Hall algebra and derived Hall algebra of a hereditary abelian category $\A$ in Section 2. In Section 3, we define a Hall algebra for the root category of $\A$ and prove its associativity. Section 4 is devoted to showing that the defined Hall algebra is isomorphic to the Drinfeld double Hall algebra of $\A$. In Section 5, as an appendix, we also define a Hall algebra for the $1$-periodic derived category of $\A$.

Throughout the paper, $k$ is a finite field with $q$ elements and $v=\sqrt{q}\in\mathbb{C}$, $\A$ is an essentially small hereditary abelian $k$-category, and $D^b(\mathcal {A})$ is the bounded derived category of $\A$ with the shift functor $[1]$. We always assume that $\A$ is finitary, i.e. for any objects $M,N\in\A$, the spaces $\Hom_{\A}(M,N)$ and $\Ext^1_{\A}(M,N)$ are both finite dimensional. Let $K(\mathcal{A})$ be the Grothendieck group of $\A$, we denote by $\hat{M}$ the image of $M$ in $K(\mathcal{A})$ for any $M\in\A$. For a finite set $S$, we denote by $|S|$ its cardinality. For an essentially small finitary category $\mathcal {E}$, we denote by ${\rm Iso}(\mathcal {E})$ the set of isomorphism classes $[X]$ of objects $X$ in $\mathcal {E}$; for each object $X\in\mathcal {E}$, denote by $\Aut X$ the automorphism group of $X$, and write $a_X$ for $|\Aut X|$. All tensor products are considered over the complex number field $\mathbb{C}$.

\section{Preliminaries}
In this section, we recall the definitions of the Ringel--Hall algebra, Drinfeld double Hall algebra and derived Hall algebra of the hereditary abelian category $\A$, and give some formulas on the (derived) Hall numbers.
\subsection{Ringel--Hall algebras}
Given objects $L,M,N \in \mathcal{A}$, we denote by $\Ext_\mathcal{A}^1(M,N)_L$ the subset of $\Ext_\mathcal{A}^1(M,N)$, which consists of the equivalence classes of short exact sequences with middle term $L$.
\begin{definition}
The \emph{Hall algebra} $\mathcal {H}(\mathcal{A})$ of $\mathcal{A}$ is the $\mathbb{C}$-vector space with the basis $\{u_{[M]}~|~[M]\in \Iso(\mathcal{A}$)\}, and with the multiplication defined by
\begin{equation}\label{jhallm}
u_{[M]} \diamond u_{[N]} = \sum\limits_{[L] \in {\rm Iso}(\mathcal{A})} {\frac{{|\Ext_\mathcal{A}^1{{(M,N)}_L}|}}{{|\Hom_\mathcal{A}(M,N)|}}} u_{[L]}.\end{equation}
\end{definition}
By \cite{R90a}, the above operation $\diamond$ defines on $\mathcal {H}(\mathcal{A})$ the structure of a unital associative algebra, and the basis element $u_{[0]}$ is the unit.
\begin{remark}
Given objects $L,M,N\in \A$, let
$g_{M,N}^L$ be the number of subobjects $X$ of $L$ such that $X\cong N$ and $L/X\cong M$.
The Riedtmann--Peng formula (cf. \cite{Riedtmann,Peng}) states that
$$g_{M,N}^{L}=\frac{|\Ext^1_{\A}(M,N)_{L}|}{|\Hom_{\A}(M,N)|}\cdot \frac{a_{L}}{a_{M}a_{N}}.$$ Thus in terms of alternative generators $\mu_{[M]}=\frac{1}{a_M}u_{[M]}$, the product (\ref{jhallm}) takes the form
$$\mu_{[M]}\diamond \mu_{[N]}= \sum\limits_{[L] \in {\rm Iso}(\mathcal{A})}g_{M,N}^L\mu_{[L]},$$
which is the definition used, for example, in \cite{R90a,Sc}.
\end{remark}

Given objects $M,N \in \mathcal{A}$, set $$\lr{M,N}:=\dim_k\Hom_{\A}(M,N)-\dim_k\Ext^1_{\A}(M,N),$$
and it descends to give a bilinear form
$$\lr{\cdot ,\cdot }: K(\mathcal{A})\times K(\mathcal{A})\longrightarrow \mathbb{Z},$$ known as the \emph{Euler form}. We also consider the \emph{symmetric Euler form}
$$(\cdot ,\cdot ): K(\mathcal{A})\times K(\mathcal{A})\longrightarrow \mathbb{Z},$$ defined by $(\alpha,\beta)=\lr{\alpha,\beta}+\lr{\beta,\alpha}$ for all $\alpha,\beta \in K(\mathcal{A})$.
The \emph{Ringel--Hall algebra} ${\mathcal {H}}_{\rm{tw}}(\mathcal{A})$ of $\mathcal{A}$ is the same vector space as $\mathcal {H}(\mathcal{A})$, but with the multiplication defined by $$u_{[M]}u_{[N]}=v^{\lr{\hat{M},\hat{N}}}\cdot u_{[M]}\diamond u_{[N]}.$$
The \emph{extended Ringel--Hall algebra} ${\mathcal {H}}_{\rm{tw}}^{\rm e}(\mathcal{A})$ of $\A$ is defined as an extension of ${\mathcal {H}}_{\rm{tw}}(\mathcal{A})$ by appending elements $K_{\alpha}$ for all $\alpha\in K(\A)$, and imposing relations $$K_{\alpha}K_{\beta}=K_{\alpha+\beta},\quad K_{\alpha}u_{[M]}=v^{(\alpha,\hat{M})}u_{[M]}K_{\alpha},$$ for $\alpha,\beta\in K(\A)$ and $[M]\in \Iso(\mathcal{A})$.

By Green \cite{Gr95} and Xiao \cite{Xiao}, the extended Ringel--Hall algebra ${\mathcal {H}}_{\rm{tw}}^{\rm e}(\mathcal{A})$
has a {topological} (cf. \cite{Sc}) bialgebra structure with the comultiplication $\Delta:{\mathcal {H}}_{\rm{tw}}^{\rm e}(\mathcal{A})\rightarrow {\mathcal {H}}_{\rm{tw}}^{\rm e}(\mathcal{A})\widehat{\otimes} {\mathcal {H}}_{\rm{tw}}^{\rm e}(\mathcal{A})$ and counit $\epsilon:{\mathcal {H}}_{\rm{tw}}^{\rm e}(\mathcal{A})\rightarrow\mathbb{C}$ defined by
$$\Delta(u_{[L]}K_{\alpha})=\sum\limits_{[M],[N] \in {\rm Iso}(\mathcal{A})}v^{\lr{\hat{M},\hat{N}}}g_{M,N}^Lu_{[M]}K_{\alpha+\hat{N}}\otimes u_{[N]}K_{\alpha}\quad\mbox{and}\quad \epsilon(u_{[L]}K_{\alpha})=\delta_{[L],[0]}.$$
Here the word {\em topological} means that the tensor product $\widehat{\otimes}$ is a {\em completed tensor product}, namely, ${\mathcal {H}}_{\rm{tw}}^{\rm e}(\mathcal{A})\widehat{\otimes} {\mathcal {H}}_{\rm{tw}}^{\rm e}(\mathcal{A})$ is the space of all formal (possibly infinite) linear combinations of elements $a\otimes b$ with $a,b\in{\mathcal {H}}_{\rm{tw}}^{\rm e}(\mathcal{A})$. If $\A$ is the category of finite dimensional nilpotent representations of a finite quiver, ${\mathcal {H}}_{\rm{tw}}^{\rm e}(\mathcal{A})$ is a genuine bialgebra. That $\Delta$ is a homomorphism of algebras amounts to the following crucial formula given in \cite{Gr95}.
\begin{theorem}{\rm(\textbf{Green's formula})}
Given objects $M,N,M',N'\in\A$, we have the following formula
\begin{equation*}
\begin{split}
&a_Ma_Na_{M'}a_{N'}\sum\limits_{[L]\in{\rm Iso}(\A)}g_{M,N}^Lg_{M',N'}^L\frac{1}{a_L}\\&=
\sum\limits_{[A],[A'],[B],[B']\in{\rm Iso}(\A)}q^{-\lr{\hat{A},\hat{B}'}}
g_{A,A'}^{M}g_{B,B'}^{N}g_{A,B}^{M'}g_{A',B'}^{N'}a_{A}a_{A'}a_{B}a_{B'}.
\end{split}
\end{equation*}
\end{theorem}

It is well known that there exists a nondegenerate symmetric bilinear
$$\varphi(-,-): {\mathcal {H}}_{\rm{tw}}^{\rm e}(\mathcal{A})\times{\mathcal {H}}_{\rm{tw}}^{\rm e}(\mathcal{A})\longrightarrow
\mathbb{C}$$ defined by
$$\varphi(u_{[M]}K_{\alpha},u_{[N]}K_{\beta})=\delta_{[M],[N]}a_{M}v^{(\alpha,\beta)},$$ which is a Hopf pairing (cf. \cite{Gr95,Sc,Xiao}).
Then the {\em Drinfeld double Hall algebra} $\mathcal {D}(\A)$ of $\A$ is by definition the free product ${\mathcal {H}}_{\rm{tw}}^{\rm e}(\mathcal{A})\ast{\mathcal {H}}_{\rm{tw}}^{\rm e}(\mathcal{A})$ (cf. \cite{Xiao,Van,Yan,LP,XuZhang}) divided out by the commutator relations
(with $a,b\in{\mathcal {H}}_{\rm{tw}}^{\rm e}(\mathcal{A})$) \begin{equation}\label{Drinfeld}\sum\varphi(a_{(2)},b_{(1)})\cdot a_{(1)}\otimes b_{(2)}=\sum\varphi(a_{(1)},b_{(2)})(1\otimes b_{(1)}) (a_{(2)}\otimes1).\end{equation} Here we use Sweedler's notation: $\Delta(a)=\sum a_{(1)}\otimes a_{(2)}$.

\subsection{Derived Hall algebras}
Given objects $M,N,X\in D^b(\A)$, set $$\{M,N\}:=\prod\limits_{i>0}|\Hom_{D^b(\A)}(M[i],N)|^{(-1)^i}.$$ In a triangulated category $\mathcal {T}$, we denote by $\Hom_{\mathcal {T}}(M,N)_X$ the subset of $\Hom_{\mathcal {T}}(M,N)$ consisting of the morphisms $M\to N$ whose cone is isomorphic to $X$. According to \cite{Toen,XiaoXu}, for any objects $X,Y,L\in D^b(\A)$, we have that
$$\frac{|\Hom_{D^b(\A)}(L,X)_{Y[1]}|}{a_X}\cdot\frac{\{L,X\}}{\{X,X\}}=
\frac{|\Hom_{D^b(\A)}(Y,L)_{X}|}{a_Y}\cdot\frac{\{Y,L\}}{\{Y,Y\}}=:F_{X,Y}^L,$$ which is called \emph{To\"en's formula}. It is easy to see that for any $X,Y,L\in\A$ we have that $F_{X,Y}^L=g_{X,Y}^L$.

\begin{definition}
The \emph{derived Hall algebra} $\mathcal {D}\mathcal {H}(\A)$ of $\A$ is the $\mathbb{C}$-vector space with the basis $\{\mu_{[X]}~|~[X]\in {\rm Iso}(D^b(\A))\}$, and with the multiplication defined by
\begin{equation}\label{dhallm}
\mu_{[X]} \mu_{[Y]}=\sum\limits_{[L]\in {\rm Iso}(D^b(\A))}F_{X,Y}^L\mu_{[L]}.\end{equation}
\end{definition}

By \cite{Toen,XiaoXu}, we know that $\mathcal {D}\mathcal {H}(\A)$ is an associative and unital algebra. The derived Riedtmann--Peng formula (cf. \cite{XiaoXu2,WWZ}) states that
$$F_{X,Y}^L=\frac{|\Ext^1_{D^b(\A)}(X,Y)_{L}|}{|\Hom_{D^b(\A)}(X,Y)|}\cdot\frac{1}{\{X,Y\}}\cdot\frac{a_L}{a_X a_Y}\cdot\frac{\{L,L\}}{\{X,X\}\{Y,Y\}},$$
where $\Ext^1_{D^b(\A)}(X,Y)_{L}:=\Hom_{D^b(\A)}(X,Y[1])_{L[1]}$. Thus in terms of alternative generators $u_{[X]}=\{X,X\}\cdot a_X\mu_{[X]}$, the product (\ref{dhallm}) takes the form
$$u_{[X]} u_{[Y]}=\sum\limits_{[L]\in {\rm Iso}(D^b(\A))}H_{X,Y}^Lu_{[L]},$$ where
$$H_{X,Y}^L=\frac{|\Ext^1_{D^b(\A)}(X,Y)_{L}|}{|\Hom_{D^b(\A)}(X,Y)|}\cdot\frac{1}{\{X,Y\}}.$$

Given objects $X_1,X_2,\cdots,X_t,L$ in $D^b(\A)$, define $F_{X_1,X_2,\cdots,X_t}^L$ to be the number such that
$$\mu_{[X_1]}\mu_{[X_2]}\cdots \mu_{[X_t]}=\sum\limits_{[L]\in {\rm Iso}(D^b(\A))}F_{X_1,X_2,\cdots,X_t}^L\mu_{[L]}.$$
By the associativity of the {derived Hall algebra} $\mathcal {D}\mathcal {H}(\A)$, we have that
\begin{equation}\label{associativity}
\begin{split}
F_{X_1,X_2,\cdots,X_t}^L&=\sum\limits_{[X]\in {\rm Iso}(D^b(\A))}F_{X_1,\cdots,X_{i-1},X}^LF_{X_i,\cdots,X_{t}}^X\\
&=\sum\limits_{[X']\in {\rm Iso}(D^b(\A))}F_{X_1,\cdots,X_{i}}^{X'}F_{X',X_{i+1},\cdots,X_{t}}^L\end{split}\end{equation} for each $1< i< t$.

By abuse of notation, in what follows, for each object $X$, we may also write $\mu_X$ and $u_X$ for $\mu_{[X]}$ and $u_{[X]}$, respectively.
The structure of the derived Hall algebra $\mathcal {D}\mathcal {H}(\A)$ has been characterized in \cite[Proposition 7.1]{Toen}. In particular, we have the following equations.
\begin{lemma}\label{diyiy}
$(1)$ For any objects $M,N\in\A$, we have that $$\mu_{M[i]}\mu_{N[j]}=q^{(-1)^{i-j}\lr{\hat{N},\hat{M}}}\mu_{N[j]}\mu_{M[i]}~\text{for~any}~i-j>1.$$

$(2)$ For any objects $M,N,X,Y\in\A$, we have that
\begin{equation*}
F_{M[1],N}^{X[1]\oplus Y}=F_{M,N[-1]}^{X\oplus Y[-1]}=q^{-\lr{\hat{Y},\hat{X}}}\frac{a_Xa_Y}{a_Ma_N}\sum\limits_{[L]\in{\rm Iso}(\A)}a_LF_{L,X}^MF_{Y,L}^N.
\end{equation*}
\end{lemma}

We remark that  all the indices $[X]$ in the following sums $\sum$ are taken over $\Iso(\A)$, until to the end of the paper.
\begin{proposition}\label{gongshi1}
Given objects $L,M,N,L_i,M_i,N_i$ in $\A$, $i=1,2$, we have that
$$F_{L_1[-1]\oplus M_1\oplus N_1[1],L_2[-1]\oplus M_2\oplus N_2[1]}^{L[-1]\oplus M\oplus N[1]}=q^{\lr{\hat{L}_2,\hat{N}_1}}\sum\limits_{[I_1],[I_3],[I_4],[I_6]}F_{N_1[1],M_2}^{I_1[1]\oplus I_3}F_{M_1,L_2[-1]}^{I_4\oplus I_6[-1]}F_{L_1,I_6}^LF_{I_4,I_3}^MF_{I_1,N_2}^N.$$
\end{proposition}
\begin{proof}
Let us consider the multiplication $\mu_{L_1[-1]\oplus M_1\oplus N_1[1]}\mu_{L_2[-1]\oplus M_2\oplus N_2[1]}$ in $\mathcal {D}\mathcal {H}(\A)$. First of all, it is easy to see that $\mu_{X[-1]\oplus Y\oplus Z[1]}=\mu_{X[-1]}\mu_{Y}\mu_{Z[1]}$ for any objects $X,Y,Z\in\A$. So,
\begin{flalign*}
\mu_{L_1[-1]\oplus M_1\oplus N_1[1]}\mu_{L_2[-1]\oplus M_2\oplus N_2[1]}
=\mu_{L_1[-1]}\mu_{M_1}\mu_{N_1[1]}\mu_{L_2[-1]}\mu_{M_2}\mu_{N_2[1]}.
\end{flalign*}
By Lemma \ref{diyiy}(1), we have that $\mu_{N_1[1]}\mu_{L_2[-1]}=q^{\lr{\hat{L}_2,\hat{N}_1}}\mu_{L_2[-1]}\mu_{N_1[1]}.$ Thus,
\begin{flalign*}
\mu_{L_1[-1]\oplus M_1\oplus N_1[1]}\mu_{L_2[-1]\oplus M_2\oplus N_2[1]}
=q^{\lr{\hat{L}_2,\hat{N}_1}}\mu_{L_1[-1]}\mu_{M_1}\mu_{L_2[-1]}\mu_{N_1[1]}\mu_{M_2}\mu_{N_2[1]}.
\end{flalign*}
Since $\mu_{M_1}\mu_{L_2[-1]}=\sum\limits_{[I_4],[I_6]}F_{M_1,L_2[-1]}^{I_4\oplus I_6[-1]}\mu_{I_6[-1]}\mu_{I_4}$
and $\mu_{N_1[1]}\mu_{M_2}=\sum\limits_{[I_1],[I_3]}F_{N_1[1],M_2}^{I_1[1]\oplus I_3}\mu_{I_3}\mu_{I_1[1]}$, we obtain that
\begin{flalign*}
&\mu_{L_1[-1]\oplus M_1\oplus N_1[1]}\mu_{L_2[-1]\oplus M_2\oplus N_2[1]}\\
&=q^{\lr{\hat{L}_2,\hat{N}_1}}\sum\limits_{[I_4],[I_6],[I_1],[I_3]}F_{M_1,L_2[-1]}^{I_4\oplus I_6[-1]}F_{N_1[1],M_2}^{I_1[1]\oplus I_3}
\mu_{L_1[-1]}\mu_{I_6[-1]}\mu_{I_4}\mu_{I_3}\mu_{I_1[1]}\mu_{N_2[1]}\\
&=q^{\lr{\hat{L}_2,\hat{N}_1}}\sum\limits_{[I_1],[I_3],[I_4],[I_6]\atop [L],[M],[N]}F_{N_1[1],M_2}^{I_1[1]\oplus I_3}F_{M_1,L_2[-1]}^{I_4\oplus I_6[-1]}F_{L_1,I_6}^LF_{I_4,I_3}^MF_{I_1,N_2}^N\mu_{I[-1]}\mu_{M}\mu_{N[1]}\\
&=q^{\lr{\hat{L}_2,\hat{N}_1}}\sum\limits_{[I_1],[I_3],[I_4],[I_6]\atop [L],[M],[N]}F_{N_1[1],M_2}^{I_1[1]\oplus I_3}F_{M_1,L_2[-1]}^{I_4\oplus I_6[-1]}F_{L_1,I_6}^LF_{I_4,I_3}^MF_{I_1,N_2}^N\mu_{I[-1]\oplus M\oplus N[1]}.
\end{flalign*}
Hence, for each fixed $L,M,N\in\A$, we obtain the desired equation.
\end{proof}
We remark that, in Proposition \ref{gongshi1}, for each triangle $$L_2[-1]\oplus M_2\oplus N_2[1]\longrightarrow L[-1]\oplus M\oplus N[1]\longrightarrow L_1[-1]\oplus M_1\oplus N_1[1]$$ in $D^b(\A)$, we have the following long exact sequence in $\A$
$$\xymatrix@R=0.7pc @C=0.6pc{0\ar[r]&N_2\ar[r]&N\ar[rr]\ar@{->>}[rd]&&N_1\ar[rr]\ar@{->>}[rd]&&M_2\ar[rr]\ar@{->>}[rd]&&M\ar[rr]\ar@{->>}[rd]&&M_1\ar[rr]\ar@{->>}[rd]&&L_2\ar[rr]\ar@{->>}[rd]&&L\ar[r]&L_1\ar[r]&0.\\
&&&I_1\ar@{>->}[ru]&&I_2\ar@{>->}[ru]&&I_3\ar@{>->}[ru]&&I_4\ar@{>->}[ru]&&I_5\ar@{>->}[ru]&&I_6\ar@{>->}[ru]&&&}$$
\begin{corollary}\label{tuilun1}
Given objects $M_1,M_2,N_1, L_2, M$ in $\A$, we have that
\begin{flalign*}F_{M_1,L_2[-1],N_1[1],M_2}^M&=q^{-\lr{\hat{L}_2,\hat{N}_1}}F_{M_1\oplus N_1[1], L_2[-1]\oplus M_2}^M\\
&=\sum\limits_{[I_3],[I_4]}\frac{a_{I_3}a_{I_4}}{a_{M_1}a_{M_2}}F_{L_2,I_4}^{M_1}F_{I_3,N_1}^{M_2}F_{I_4,I_3}^M.\end{flalign*}
\end{corollary}
\begin{proof}
Noting that \begin{flalign*}\mu_{M_1}\mu_{L_2[-1]}\mu_{N_1[1]}\mu_{M_2}&=q^{-\lr{\hat{L}_2,\hat{N}_1}}\mu_{M_1}\mu_{N_1[1]}\mu_{L_2[-1]}\mu_{M_2}\\
&=q^{-\lr{\hat{L}_2,\hat{N}_1}}\mu_{M_1\oplus N_1[1]}\mu_{L_2[-1]\oplus M_2},
\end{flalign*}
we obtain the first equation.
Taking $L=L_1=0$ and $N=N_2=0$ in Proposition \ref{gongshi1}, we have that $I_1=I_6=0$, and thus
\begin{flalign*}F_{M_1\oplus N_1[1], L_2[-1]\oplus M_2}^M&=q^{\lr{\hat{L}_2,\hat{N}_1}}\sum\limits_{[I_3],[I_4]}F_{N_1[1],M_2}^{I_3}F_{M_1,L_2[-1]}^{I_4}F_{I_4,I_3}^M\\
&=q^{\lr{\hat{L}_2,\hat{N}_1}}\sum\limits_{[I_3],[I_4]}\frac{a_{I_3}}{a_{M_2}}F_{I_3,N_1}^{M_2}\frac{a_{I_4}}{a_{M_1}}F_{L_2,I_4}^{M_1}F_{I_4,I_3}^M.\end{flalign*}
Therefore, we finish the proof.
\end{proof}
\begin{corollary}\label{tuilun2.7}
Given objects $M_1,M_2,N_1, L_2, M$ in $\A$, we have that
\begin{flalign*}H_{M_1\oplus N_1[1], L_2[-1]\oplus M_2}^M&=q^{-\lr{\hat{M}_1,\hat{N}_1}-\lr{\hat{L}_2,\hat{M}_2}}\frac{a_{M_1}a_{M_2}a_{N_1}a_{L_2}}{a_M}F_{M_1\oplus N_1[1],L_2[-1]\oplus M_2}^M\\
&=q^{\lr{\hat{L}_2,\hat{N}_1}-\lr{\hat{M}_1,\hat{N}_1}-\lr{\hat{L}_2,\hat{M}_2}}\frac{a_{M_1}a_{M_2}a_{N_1}a_{L_2}}{a_M}F_{M_1,L_2[-1],N_1[1],M_2}^M.\end{flalign*}
\end{corollary}
\begin{proof}
By Corollary \ref{tuilun1}, we only need to prove the first equation. By definition,
{\small\begin{flalign*}&H_{M_1\oplus N_1[1], L_2[-1]\oplus M_2}^M=\\&\frac{a_{M_1\oplus N_1[1]}a_{L_2[-1]\oplus M_2}}{a_M}\frac{\{M_1\oplus N_1[1],M_1\oplus N_1[1]\}\{L_2[-1]\oplus M_2,L_2[-1]\oplus M_2\}}{\{M,M\}}F_{M_1\oplus N_1[1], L_2[-1]\oplus M_2}^M\end{flalign*}}
and it is easy to see that $$a_{M_1\oplus N_1[1]}=a_{M_1}a_{N_1}|\Ext_{\A}^1(M_1,N_1)|,~a_{L_2[-1]\oplus M_2}=a_{L_2}a_{M_2}|\Ext_{\A}^1(L_2,M_2)|,$$
$$\{M,M\}=1,~\{M_1\oplus N_1[1],M_1\oplus N_1[1]\}=\frac{1}{|{\rm Hom}_{\A}(M_1,N_1)|},$$
and $$\{L_2[-1]\oplus M_2,L_2[-1]\oplus M_2\}=\frac{1}{|{\rm Hom}_{\A}(L_2,M_2)|}.$$ Combining these, we complete the proof.
\end{proof}

\section{Hall algebras for root categories}
In this section, we define a Hall algebra for the root category of $\A$ and prove its associativity. Let $G$ be an automorphism of $D^b(\A)$. We recall that the {\em orbit category} $D^b(\A)/G$ is the category whose objects are
by definition the $G$-orbits $\widetilde{X}$ of objects $X$ in $D^b(\A)$ and morphisms are given by
$$\Hom_{D^b(\A)/G}(\widetilde{X},\widetilde{Y})=\bigoplus\limits_{i\in\mathbb{Z}}\Hom_{D^b(\A)}(X,G^iY).$$
If $G$ is the shift functor $[1]$, the corresponding orbit category is called the {\em $1$-periodic derived category} of $\A$, which is also denoted by $\mathcal {D}_1(\A)$.
If $G$ is the shift functor $[2]$, the corresponding orbit category is called the {\em root category} of $\A$, which is also denoted by $\mathcal {R}(\A)$.
The categories $\mathcal {D}_1(\A)$ and $\mathcal {R}(\A)$ are both triangulated categories (cf. \cite{RingelZ,PX97,Kellero}), and their suspension functors are also denoted by $[1]$. The canonical covering functors $F:D^b(\A)\rightarrow D^b(\A)/G$ for $G=[1],[2]$ are triangulated functors, and the abelian category $\A$ can be fully embedded into $\mathcal {R}(\A)$.

The following lemma is needed in the sequel.
\begin{lemma}\label{triangles}
For any $5$-term exact sequence $0\rightarrow K\rightarrow X\stackrel{f_1}\rightarrow Y\stackrel{f_2}\rightarrow Z\rightarrow C\rightarrow 0$ in $\A$, there exist $\delta_1,\delta_2$ such that we have the following triangle in $D^b(\A)${\rm:}
$$\xymatrix{X\oplus C[-1]\ar[r]^-{(f_1,\delta_1)}& Y\ar[r]^-{\delta_2\choose f_2}& K[1]\oplus Z.}$$
\end{lemma}
\begin{proof}
First of all, we have the following commutative diagram in $\A$
$$\xymatrix@R=0.7pc{0\ar[r]&K\ar[r]^-{\lambda}& X\ar[rr]^-{f_1}\ar@{->>}[rd]_-{\pi_1}&& Y\ar[rr]^-{f_2}\ar@{->>}[rd]_-{\pi_2}&&Z\ar[r]^{p}& C\ar[r]& 0,\\
&&&I_1\ar@{>->}[ru]_-{i_1}&&I_2\ar@{>->}[ru]_-{i_2}&&&}$$ where $I_1=\im f_1$ and $I_2=\im f_2$.

By the octahedral axiom in triangulated categories, we have the following commutative diagram of triangles in $D^b(\A)$
$$\xymatrix{Y\ar[r]^-{\pi_2}\ar@{=}[d]&I_2\ar[r]\ar[d]^{i_2}&I_1[1]\ar@{-->}[d]\\
Y\ar[r]^-{f_2}&Z\ar[r]\ar[d]^-{p}&\cone(f_2)\ar@{-->}[d]\\
&C\ar@{=}[r]&C.}$$
Since $\A$ is hereditary, $\Hom_{D^b(\A)}(C,I_1[2])=0$. Thus, $\cone(f_2)\cong I_1[1]\oplus C$ in $D^b(\A)$. So we have the following commutative diagram of triangles in $D^b(\A)$
$$\xymatrix@C=4pc {Y\ar[r]^-{\pi_2}\ar@{=}[d]&I_2\ar[r]\ar[d]^{i_2}&I_1[1]\ar@{-->}[d]^-{{\rho[1]\choose 0}}\ar[r]^{-i_1[1]}&Y[1]\ar@{=}[d]\\
Y\ar[r]^-{f_2}&Z\ar[r]\ar[d]^-{p}&I_1[1]\oplus C\ar@{-->}[d]\ar[r]^-{(-i'_1[1],-\delta'_1)}&Y[1]\\
&C\ar@{=}[r]&C&}$$ such that $i_1=i'_1\rho$ and $\rho:I_1\rightarrow I_1$ is an isomorphism.
Taking $\delta_1=\delta'_1[-1]$, we obtain a triangle in $D^b(\A)$
$$\xymatrix{I_1\oplus C[-1]\ar[r]^-{(i'_1,\delta_1)}&Y\ar[r]^-{f_2}&Z.}$$ Since $\rho$ is an isomorphism, it is easy to see that
$\xymatrix{I_1\oplus C[-1]\ar[r]^-{(i_1,\delta_1)}&Y\ar[r]^-{f_2}&Z}$ is also a triangle in $D^b(\A)$. Considering the following commutative diagram of triangles in $D^b(\A)$
$$\xymatrix{X\oplus C[-1]\ar[r]^-{\left({\begin{smallmatrix}{\pi_1}&\\&1\end{smallmatrix}}\right)}\ar@{=}[d]&I_1\oplus C[-1]\ar[r]\ar[d]^{(i_1,\delta_1)}&K[1]\ar@{-->}[d]\\
X\oplus C[-1]\ar[r]^-{(f_1,\delta_1)}&Y\ar[r]\ar[d]^-{f_2}&\cone(f_1,\delta_1)\ar@{-->}[d]\\
&Z\ar@{=}[r]&Z,}$$
we conclude that $\cone(f_1,\delta_1)\cong K[1]\oplus Z$ in $D^b(\A)$, since $\Hom_{D^b(\A)}(Z,K[2])=0$.
Hence, we have the following commutative diagram of triangles in $D^b(\A)$
$$\xymatrix{X\oplus C[-1]\ar[r]^-{\left({\begin{smallmatrix}{\pi_1}&\\&1\end{smallmatrix}}\right)}\ar@{=}[d]&I_1\oplus C[-1]\ar[r]\ar[d]^{(i_1,\delta_1)}&K[1]\ar@{-->}[d]\\
X\oplus C[-1]\ar[r]^-{(f_1,\delta_1)}&Y\ar[r]^-{{\delta_2\choose x}}\ar[d]^-{f_2}&K[1]\oplus Z\ar@{-->}[d]^-{(0,\eta)}\\
&Z\ar@{=}[r]&Z}$$ such that $f_2=\eta x$ and $\eta$ is an isomorphism. Therefore, we obtain that $$\xymatrix{X\oplus C[-1]\ar[r]^-{(f_1,\delta_1)}& Y\ar[r]^-{\delta_2\choose f_2}& K[1]\oplus Z}$$ is a triangle in $D^b(\A)$.
\end{proof}

Now, let us compare the extensions in the root category $\mathcal {R}(\A)$ with those in the derived category $D^b(\A)$.
Given objects $A_i, B_i\in\A$ with $i=0,1$, we have the following isomorphisms
\begin{flalign*}
\Ext^1_{\mathcal {R}(\A)}(\widetilde{A}_0\oplus \widetilde{A}_1[1],\widetilde{B}_0\oplus \widetilde{B}_1[1])&=\Hom_{\mathcal {R}(\A)}(\widetilde{A}_0\oplus \widetilde{A}_1[1],\widetilde{B}_0[1]\oplus \widetilde{B}_1[2])\\
&\cong\Hom_{\mathcal {R}(\A)}(\widetilde{A}_0\oplus \widetilde{A}_1[1],\widetilde{B}_0[1]\oplus \widetilde{B}_1)\\
&\cong\bigoplus\limits_{i\in\mathbb{Z}}\Hom_{D^b(\A)}(A_0\oplus A_1[1],B_0[2i+1]\oplus B_1[2i]).
\end{flalign*}
Since $\A$ is hereditary, we obtain that
\begin{equation}\label{tonggou}
\begin{split}
&\Ext^1_{\mathcal {R}(\A)}(\widetilde{A}_0\oplus \widetilde{A}_1[1],\widetilde{B}_0\oplus \widetilde{B}_1[1])\\&\cong\Ext^1_{\A}(A_0,B_0)\oplus \Hom_{\A}(A_0,B_1)\oplus \Ext^1_{\A}(A_1,B_1)\oplus\Hom_{\A}(A_1,B_0).
\end{split}
\end{equation}

In fact, for any triangle $\widetilde{B}_0\oplus \widetilde{B}_1[1]\longrightarrow \widetilde{M}_0\oplus \widetilde{M}_1[1]\longrightarrow \widetilde{A}_0\oplus \widetilde{A}_1[1]$ in $\mathcal {R}(\A)$ such that $A_i,B_i,M_i\in\A$ with $i=0,1$, we have the following cyclic exact sequence in $\A$
$$\cdots\rightarrow B_0\rightarrow M_0\rightarrow A_0\stackrel{g}\rightarrow B_1\rightarrow M_1\rightarrow A_1\stackrel{f}\rightarrow B_0\rightarrow M_0\rightarrow A_0\rightarrow \cdots.$$
Setting $\im f=I_0$ and $\im g=I_1$, we have the following long exact sequence in $\A$
\begin{equation}\label{xyczhl}
\xymatrix@R=0.8pc @C=0.6pc{0\ar[r]&I_0\ar[r]&B_0\ar[r]&M_0\ar[r]&A_0\ar[rr]^-{g}\ar@{->>}[rd]&&B_1\ar[r]&M_1\ar[r]&A_1\ar[rr]^-{f}\ar@{->>}[rd]&&B_0\ar[r]&M_0\ar[r]&A_0\ar[r]&I_1\ar[r]&0.\\
&&&&&I_1\ar@{>->}[ru]&&&&I_0\ar@{>->}[ru]&&&&&}\end{equation}
Thus, by Lemma \ref{triangles}, we obtain two triangles in $D^b(\A)$
\begin{equation}\label{sanjiao1}
B_0\oplus I_1[-1]\longrightarrow M_0\longrightarrow I_0[1]\oplus A_0
\end{equation} and
\begin{equation}\label{sanjiao2}
B_1\oplus I_0[-1]\longrightarrow M_1\longrightarrow I_1[1]\oplus A_1.
\end{equation}

Conversely, given triangles (\ref{sanjiao1}) and (\ref{sanjiao2}) in $D^b(\A)$, by considering the long exact sequences in homologies and splicing them, we also have a long exact sequence (\ref{xyczhl}).

On the other hand,
$$\Ext^1_{D^b(\A)}(I_0[1]\oplus A_0,B_0\oplus I_1[-1])\cong\Ext_{\A}^1(A_0,B_0)\oplus\Hom_{\A}(A_0,I_1)\oplus\Hom_{\A}(I_0,B_0)$$
and
$$\Ext^1_{D^b(\A)}(I_1[1]\oplus A_1,B_1\oplus I_0[-1])\cong\Ext_{\A}^1(A_1,B_1)\oplus\Hom_{\A}(A_1,I_0)\oplus\Hom_{\A}(I_1,B_1).$$
Note that the elements in $\Ext^1_{D^b(\A)}(I_0[1]\oplus A_0,B_0\oplus I_1[-1])$ consist of the equivalence classes of the triangles in $D^b(\A)$ of the following form
$$B_0\oplus I_1[-1]\longrightarrow C[-1]\oplus M\oplus K[1]\longrightarrow I_0[1]\oplus A_0$$ with $C,M,K\in\A$, where $C,K$ may not be zero. Hence, in general,
$$\bigoplus\limits_{[I_0],[I_1]\in{\rm Iso}(\A)}(\Ext^1_{D^b(\A)}(I_0[1]\oplus A_0,B_0\oplus I_1[-1])\oplus\Ext^1_{D^b(\A)}(I_1[1]\oplus A_1,B_1\oplus I_0[-1]))$$
is larger than $\Ext^1_{\mathcal {R}(\A)}(\widetilde{A}_0\oplus \widetilde{A}_1[1],\widetilde{B}_0\oplus \widetilde{B}_1[1])$. However, we have the following equation.
\begin{proposition}
For any objects $A_i, B_i, M_i\in\A$ with $i=0,1$, we have that
\begin{equation*}
\begin{split}
&|\Ext^1_{\mathcal {R}(\A)}(\widetilde{A}_0\oplus \widetilde{A}_1[1],\widetilde{B}_0\oplus \widetilde{B}_1[1])_{\widetilde{M}_0\oplus \widetilde{M}_1[1]}|=
\\&\sum\limits_{[I_0],[I_1]\in{\rm Iso}(\A)}\frac{|\Ext^1_{D^b(\A)}(I_0[1]\oplus A_0,B_0\oplus I_1[-1])_{M_0}|\cdot|\Ext^1_{D^b(\A)}(I_1[1]\oplus A_1,B_1\oplus I_0[-1])_{M_1}|}{a_{I_0}a_{I_1}}.
\end{split}
\end{equation*}
\end{proposition}
\begin{proof}
By relating the equivalence classes in three extension subsets with the equivalence classes of long exact sequences under the corresponding automorphism group actions, and using the above discussions, we can easily complete the proof.
\end{proof}

Hence, we define a Hall algebra associated to the root category $\mathcal {R}(\A)$ as follows.
\begin{definition}\label{maindef}
The Hall algebra $\mathcal {D}\mathcal {H}_2(\A)$, called the {\em $2$-periodic extended derived Hall algebra} of $\A$, is the $\mathbb{C}$-vector space with the basis $$\{K_{\alpha}K_{\beta}^\ast u_{M_0\oplus M_1[1]}~|~\alpha,\beta\in K(\A),[M_0],[M_1]\in {\rm Iso}(\A)\},$$ and with the multiplication defined on basis elements by
{\small\begin{flalign*}&(K_{\alpha_0}K_{\alpha_1}^*u_{A_0\oplus A_1[1]})(K_{\beta_0}K_{\beta_1}^*u_{B_0\oplus B_1[1]})=\\&
v^{a_0}\sum\limits_{[I_0],[I_1],[M_0],[M_1]\in{\rm Iso}(\A)}v^{\lr{\hat{I}_1-\hat{I}_0,\hat{A}_0-\hat{A}_1+\hat{B}_0-\hat{B}_1}}
\frac{H_{I_0[1]\oplus A_0,B_0\oplus I_1[-1]}^{M_0}H_{I_1[1]\oplus A_1,B_1\oplus I_0[-1]}^{M_1}}{a_{I_0}a_{I_1}}\\
&\quad\quad \quad\quad\quad\quad\quad\quad\quad\quad\quad\quad K_{\alpha_0+\beta_0+\hat{I}_0}K_{\alpha_1+\beta_1+\hat{I}_1}^\ast u_{M_0\oplus M_1[1]},
\end{flalign*}}
where $a_0=(\beta_1-\beta_0,\hat{A}_0-\hat{A}_1)+\lr{\hat{A}_0,\hat{B}_0}+\lr{\hat{A}_1,\hat{B}_1}$.
\end{definition}
\begin{remark}\label{dyjx}
$(1)$~In Definition \ref{maindef}, we have the following equations
\begin{flalign}
&K_{\alpha_0}K_{\beta_0}=K_{\alpha_0+\beta_0},~K_{\alpha_1}^\ast K_{\beta_1}^\ast=K_{\alpha_1+\beta_1}^\ast,~K_{\alpha_1}^*K_{\beta_0}=K_{\beta_0}K_{\alpha_1}^*,\\
&u_{A_0\oplus A_1[1]}K_{\beta_0}=v^{-(\beta_0,\hat{A}_0-\hat{A}_1)}K_{\beta_0}u_{A_0\oplus A_1[1]},\\
&u_{A_0\oplus A_1[1]}K_{\beta_1}^*=v^{(\beta_1,\hat{A}_0-\hat{A}_1)}K_{\beta_1}^*u_{A_0\oplus A_1[1]}.
\end{flalign}

$(2)$~In Definition \ref{maindef}, since $\hat{M}_0=\hat{A}_0+\hat{B}_0-\hat{I}_0-\hat{I}_1$ and $\hat{M}_1=\hat{A}_1+\hat{B}_1-\hat{I}_0-\hat{I}_1$, we have that
$\hat{M}_0-\hat{M}_1=\hat{A}_0-\hat{A}_1+\hat{B}_0-\hat{B}_1$. Hence, the Hall algebra $\mathcal {D}\mathcal {H}_2(\A)$ is a $K(\A)$-graded algebra with the grade of $K_{\alpha}K_{\beta}^\ast u_{M_0\oplus M_1[1]}$ defined to be $\hat{M}_0-\hat{M}_1$.
\end{remark}

\begin{theorem}
The $2$-periodic extended derived Hall algebra $\mathcal {D}\mathcal {H}_2(\A)$ is an associative algebra.
\end{theorem}
\begin{proof}
Since $\mathcal {D}\mathcal {H}_2(\A)$ is $K(\A)$-graded, it is easy to see that it suffices to prove the equation $$(u_{A_0\oplus A_1[1]}u_{B_0\oplus B_1[1]})u_{C_0\oplus C_1[1]}=u_{A_0\oplus A_1[1]}(u_{B_0\oplus B_1[1]}u_{C_0\oplus C_1[1]})$$ without involving $K$-elements, for any $[A_i],[B_i],[C_i]\in{\rm Iso}(\A), i=0,1$.

On the one hand,
\begin{flalign*}&(u_{A_0\oplus A_1[1]}u_{B_0\oplus B_1[1]})u_{C_0\oplus C_1[1]}=\\&
\sum\limits_{[I_0],[I_1], [X_0],[X_1]}v^a\frac{H_{I_0[1]\oplus A_0,B_0\oplus I_1[-1]}^{X_0}H_{I_1[1]\oplus A_1,B_1\oplus I_0[-1]}^{X_1}}{a_{I_0}a_{I_1}}K_{\hat{I}_0}K_{\hat{I}_1}^\ast u_{X_0\oplus X_1[1]}u_{C_0\oplus C_1[1]},\end{flalign*}
where $a=\lr{\hat{I}_1-\hat{I}_0,\hat{A}_0-\hat{A}_1+\hat{B}_0-\hat{B}_1}+\lr{\hat{A}_0,\hat{B}_0}+\lr{\hat{A}_1,\hat{B}_1}$.
Thus, \begin{flalign*}&(u_{A_0\oplus A_1[1]}u_{B_0\oplus B_1[1]})u_{C_0\oplus C_1[1]}=\\&
\sum\limits_{[I_0],[I_1], [X_0],[X_1]\atop [J_0],[J_1],[M_0],[M_1]}v^b\frac{H_{I_0[1]\oplus A_0,B_0\oplus I_1[-1]}^{X_0}H_{I_1[1]\oplus A_1,B_1\oplus I_0[-1]}^{X_1}}{a_{I_0}a_{I_1}}\frac{H_{J_0[1]\oplus X_0,C_0\oplus J_1[-1]}^{M_0}H_{J_1[1]\oplus X_1,C_1\oplus J_0[-1]}^{M_1}}{a_{J_0}a_{J_1}}\\&\quad\quad\quad\quad\quad\quad K_{\hat{I}_0+\hat{J}_0}K_{\hat{I}_1+\hat{J}_1}^\ast u_{M_0\oplus M_1[1]},\end{flalign*}
where $b=a+\lr{\hat{J}_1-\hat{J}_0,\hat{X}_0-\hat{X}_1+\hat{C}_0-\hat{C}_1}+\lr{\hat{X}_0,\hat{C}_0}+\lr{\hat{X}_1,\hat{C}_1}$.

Using Corollary \ref{tuilun2.7}, we obtain that
\begin{flalign*}&(u_{A_0\oplus A_1[1]}u_{B_0\oplus B_1[1]})u_{C_0\oplus C_1[1]}=\\&
\sum\limits_{[I_0],[I_1], [X_0],[X_1]\atop [J_0],[J_1],[M_0],[M_1]}v^bq^c\frac{a_{A_0}a_{B_0}a_{I_0}a_{I_1}}{a_{X_0}}F_{A_0,I_1[-1],I_0[1],B_0}^{X_0}
\frac{a_{A_1}a_{B_1}a_{I_0}a_{I_1}}{a_{X_1}}F_{A_1,I_0[-1],I_1[1],B_1}^{X_1}\frac{1}{a_{I_0}a_{I_1}}\\
&\quad\quad\quad\quad\frac{a_{X_0}a_{C_0}a_{J_0}a_{J_1}}{a_{M_0}}F_{X_0,J_1[-1],J_0[1],C_0}^{M_0}
\frac{a_{X_1}a_{C_1}a_{J_0}a_{J_1}}{a_{M_1}}F_{X_1,J_0[-1],J_1[1],C_1}^{M_1}\frac{1}{a_{J_0}a_{J_1}}\\&
\quad\quad\quad\quad K_{\hat{I}_0+\hat{J}_0}K_{\hat{I}_1+\hat{J}_1}^\ast u_{M_0\oplus M_1[1]},\end{flalign*}
where $c=\lr{\hat{I}_1,\hat{I}_0}-\lr{\hat{A}_0,\hat{I}_0}-\lr{\hat{I}_1,\hat{B}_0}+\lr{\hat{I}_0,\hat{I}_1}-\lr{\hat{A}_1,\hat{I}_1}-\lr{\hat{I}_0,\hat{B}_1}
+\lr{\hat{J}_1,\hat{J}_0}-\lr{\hat{X}_0,\hat{J}_0}-\lr{\hat{J}_1,\hat{C}_0}+\lr{\hat{J}_0,\hat{J}_1}-\lr{\hat{X}_1,\hat{J}_1}-\lr{\hat{J}_0,\hat{C}_1}$.

Using Corollary \ref{tuilun1}, we obtain that
\begin{flalign*}&(u_{A_0\oplus A_1[1]}u_{B_0\oplus B_1[1]})u_{C_0\oplus C_1[1]}=\\&
a_{A_0}a_{B_0}a_{C_0}a_{A_1}a_{B_1}a_{C_1}\sum\limits_{[I_0],[I_1], [X_0],[X_1]\atop [J_0],[J_1],[M_0],[M_1]}v^bq^c\frac{a_{I_0}a_{I_1}a_{J_0}a_{J_1}}{a_{M_0}a_{M_1}}\\&
(\sum\limits_{[N_0],[L_0]}\frac{a_{N_0}a_{L_0}}{a_{A_0}a_{B_0}}F_{I_1,N_0}^{A_0}F_{L_0,I_0}^{B_0}F_{N_0,L_0}^{X_0})
(\sum\limits_{[N_1],[L_1]}\frac{a_{N_1}a_{L_1}}{a_{A_1}a_{B_1}}F_{I_0,N_1}^{A_1}F_{L_1,I_1}^{B_1}F_{N_1,L_1}^{X_1})\\&
(\sum\limits_{[S_0],[T_0]}\frac{a_{S_0}a_{T_0}}{a_{X_0}a_{C_0}}F_{J_1,S_0}^{X_0}F_{T_0,J_0}^{C_0}F_{S_0,T_0}^{M_0})
(\sum\limits_{[S_1],[T_1]}\frac{a_{S_1}a_{T_1}}{a_{X_1}a_{C_1}}F_{J_0,S_1}^{X_1}F_{T_1,J_1}^{C_1}F_{S_1,T_1}^{M_1})
\\&\quad\quad\quad\quad K_{\hat{I}_0+\hat{J}_0}K_{\hat{I}_1+\hat{J}_1}^\ast u_{M_0\oplus M_1[1]}.\end{flalign*}

Noting that $\hat{X}_0=\hat{A}_0+\hat{B}_0-\hat{I}_0-\hat{I}_1$ and $\hat{X}_1=\hat{A}_1+\hat{B}_1-\hat{I}_0-\hat{I}_1$, we have that
\begin{flalign*}&(u_{A_0\oplus A_1[1]}u_{B_0\oplus B_1[1]})u_{C_0\oplus C_1[1]}=\\&
\sum\limits_{[I_0],[I_1],[J_0],[J_1],[M_0], [M_1], [N_0]\atop [L_0],[N_1],[L_1],[S_0],[T_0],[S_1],[T_1]}v^bq^c\frac{a_{I_0}a_{I_1}a_{J_0}a_{J_1}a_{N_0}a_{L_0}a_{N_1}a_{L_1}a_{S_0}a_{T_0}a_{S_1}a_{T_1}}{a_{M_0}a_{M_1}}\\&
F_{I_1,N_0}^{A_0}F_{L_0,I_0}^{B_0}F_{T_0,J_0}^{C_0}F_{I_0,N_1}^{A_1}F_{L_1,I_1}^{B_1}F_{T_1,J_1}^{C_1}F_{S_0,T_0}^{M_0}F_{S_1,T_1}^{M_1}\sum\limits_{[X_0]}F_{N_0,L_0}^{X_0}F_{J_1,S_0}^{X_0}\frac{1}{a_{X_0}}\\&
\sum\limits_{[X_1]}F_{N_1,L_1}^{X_1}F_{J_0,S_1}^{X_1}\frac{1}{a_{X_1}}
K_{\hat{I}_0+\hat{J}_0}K_{\hat{I}_1+\hat{J}_1}^\ast u_{M_0\oplus M_1[1]}.\end{flalign*}

Using Green's formula, we have that
\begin{flalign*}
\sum\limits_{[X_0]}F_{N_{0},L_{0}}^{X_0}F_{J_{1},S_0}^{X_0}\frac{1}{a_{X_0}}=\sum\limits_{[U_0],[V_0],[W_0],[Z_0]}
q^{-\lr{\hat{U}_0,\hat{Z}_0}}F_{U_0,V_0}^{N_0}F_{W_0,Z_0}^{L_0}F_{U_0,W_0}^{J_{1}}F_{V_0,Z_0}^{S_0}\frac{a_{U_0}a_{V_0}a_{W_0}a_{Z_0}}{a_{N_0}a_{L_0}a_{J_{1}}a_{S_0}}
\end{flalign*}
and
\begin{flalign*}
\sum\limits_{[X_1]}F_{N_{1},L_{1}}^{X_1}F_{J_{0},S_1}^{X_1}\frac{1}{a_{X_1}}=\sum\limits_{[U_1],[V_1],[W_1],[Z_1]}
q^{-\lr{\hat{U}_1,\hat{Z}_1}}F_{U_1,V_1}^{N_1}F_{W_1,Z_1}^{L_1}F_{U_1,W_1}^{J_{0}}F_{V_1,Z_1}^{S_1}\frac{a_{U_1}a_{V_1}a_{W_1}a_{Z_1}}{a_{N_1}a_{L_1}a_{J_{0}}a_{S_1}}.
\end{flalign*}

Hence, we obtain that
\begin{flalign*}&(u_{A_0\oplus A_1[1]}u_{B_0\oplus B_1[1]})u_{C_0\oplus C_1[1]}=\\&
\sum\limits_{[I_0],[I_1],[J_0],[J_1],[M_0], [M_1], [N_0], [L_0],[N_1],[L_1],[S_0]\atop [T_0],[S_1],[T_1],[U_0],[V_0],[W_0],[Z_0],[U_1],[V_1],[W_1],[Z_1]}v^bq^{c'}\frac{a_{I_0}a_{I_1}a_{T_0}a_{T_1}a_{U_0}a_{V_0}a_{W_0}a_{Z_0}a_{U_1}a_{V_1}a_{W_1}a_{Z_1}}{a_{M_0}a_{M_1}}\\&
\quad\quad\quad\quad\quad  F_{I_1,N_0}^{A_0}F_{L_0,I_0}^{B_0}F_{T_0,J_0}^{C_0}F_{I_0,N_1}^{A_1}F_{L_1,I_1}^{B_1}
F_{T_1,J_1}^{C_1}F_{S_0,T_0}^{M_0}F_{S_1,T_1}^{M_1}F_{U_0,V_0}^{N_0}F_{W_0,Z_0}^{L_0}\\&
\quad\quad\quad\quad\quad  F_{U_0,W_0}^{J_{1}}F_{V_0,Z_0}^{S_0}F_{U_1,V_1}^{N_1}F_{W_1,Z_1}^{L_1}F_{U_1,W_1}^{J_{0}}F_{V_1,Z_1}^{S_1}
K_{\hat{I}_0+\hat{J}_0}K_{\hat{I}_1+\hat{J}_1}^\ast u_{M_0\oplus M_1[1]},\end{flalign*}
where $c'=c-\lr{\hat{U}_0,\hat{Z}_0}-\lr{\hat{U}_1,\hat{Z}_1}$.

Noting that $\hat{J}_0=\hat{U}_1+\hat{W}_1$, $\hat{J}_1=\hat{U}_0+\hat{W}_0$ and
applying the associativity formula (\ref{associativity}),
we get that
\begin{equation}\label{zuobian}
\begin{split}&(u_{A_0\oplus A_1[1]}u_{B_0\oplus B_1[1]})u_{C_0\oplus C_1[1]}=\\&
\sum\limits_{[I_0],[I_1],[M_0], [M_1], [T_0],[T_1],[U_0]\atop[V_0],[W_0],[Z_0],[U_1],[V_1],[W_1],[Z_1]}v^bq^{c'}\frac{a_{I_0}a_{I_1}a_{T_0}a_{T_1}a_{U_0}a_{V_0}a_{W_0}a_{Z_0}a_{U_1}a_{V_1}a_{W_1}a_{Z_1}}{a_{M_0}a_{M_1}}\\&
F_{I_1,U_0,V_0}^{A_0}F_{W_0,Z_0,I_0}^{B_0}F_{T_0,U_1,W_1}^{C_0}F_{I_0,U_1,V_1}^{A_1}F_{W_1,Z_1,I_1}^{B_1}F_{T_1,U_0,W_0}^{C_1}F_{V_0,Z_0,T_0}^{M_0}F_{V_1,Z_1,T_1}^{M_1}\\&
\quad\quad\quad\quad\quad\quad\quad\quad\quad\quad K_{\hat{I}_0+\hat{U}_1+\hat{W}_1}K_{\hat{I}_1+\hat{U}_0+\hat{W}_0}^\ast u_{M_0\oplus M_1[1]}.\end{split}\end{equation}

On the other hand,
\begin{flalign*}&u_{A_0\oplus A_1[1]}(u_{B_0\oplus B_1[1]}u_{C_0\oplus C_1[1]})=\\&
\sum\limits_{[J_0],[J_1], [Y_0],[Y_1]}v^{x_0}\frac{H_{J_0[1]\oplus B_0,C_0\oplus J_1[-1]}^{Y_0}H_{J_1[1]\oplus B_1,C_1\oplus J_0[-1]}^{Y_1}}{a_{J_0}a_{J_1}}u_{A_0\oplus A_1[1]}K_{\hat{J}_0}K_{\hat{J}_1}^\ast u_{Y_0\oplus Y_1[1]},\end{flalign*}
where $x_0=\lr{\hat{J}_1-\hat{J}_0,\hat{B}_0-\hat{B}_1+\hat{C}_0-\hat{C}_1}+\lr{\hat{B}_0,\hat{C}_0}+\lr{\hat{B}_1,\hat{C}_1}$.
Noting that $$u_{A_0\oplus A_1[1]}K_{\hat{J}_0}K_{\hat{J}_1}^\ast=v^{(\hat{J}_1-\hat{J}_0,\hat{A}_0-\hat{A}_1)}K_{\hat{J}_0}K_{\hat{J}_1}^\ast u_{A_0\oplus A_1[1]},$$
we obtain that
\begin{flalign*}&u_{A_0\oplus A_1[1]}(u_{B_0\oplus B_1[1]}u_{C_0\oplus C_1[1]})=\\&
\sum\limits_{[J_0],[J_1], [Y_0],[Y_1]}v^{x}\frac{H_{J_0[1]\oplus B_0,C_0\oplus J_1[-1]}^{Y_0}H_{J_1[1]\oplus B_1,C_1\oplus J_0[-1]}^{Y_1}}{a_{J_0}a_{J_1}}K_{\hat{J}_0}K_{\hat{J}_1}^\ast u_{A_0\oplus A_1[1]} u_{Y_0\oplus Y_1[1]},\end{flalign*}
where $x=x_0+(\hat{J}_1-\hat{J}_0,\hat{A}_0-\hat{A}_1)$.
Thus, \begin{flalign*}&u_{A_0\oplus A_1[1]}(u_{B_0\oplus B_1[1]}u_{C_0\oplus C_1[1]})=\\&
\sum\limits_{[J_0],[J_1], [Y_0],[Y_1]\atop [I_0],[I_1],[M_0],[M_1]}v^y\frac{H_{J_0[1]\oplus B_0,C_0\oplus J_1[-1]}^{Y_0}H_{J_1[1]\oplus B_1,C_1\oplus J_0[-1]}^{Y_1}}{a_{J_0}a_{J_1}}\frac{H_{I_0[1]\oplus A_0,Y_0\oplus I_1[-1]}^{M_0}H_{I_1[1]\oplus A_1,Y_1\oplus I_0[-1]}^{M_1}}{a_{I_0}a_{I_1}}\\&\quad\quad\quad\quad\quad\quad K_{\hat{I}_0+\hat{J}_0}K_{\hat{I}_1+\hat{J}_1}^\ast u_{M_0\oplus M_1[1]},\end{flalign*}
where $y=x+\lr{\hat{I}_1-\hat{I}_0,\hat{A}_0-\hat{A}_1+\hat{Y}_0-\hat{Y}_1}+\lr{\hat{A}_0,\hat{Y}_0}+\lr{\hat{A}_1,\hat{Y}_1}$.

Using Corollary \ref{tuilun2.7}, we obtain that
\begin{flalign*}&u_{A_0\oplus A_1[1]}(u_{B_0\oplus B_1[1]}u_{C_0\oplus C_1[1]})=\\&
\sum\limits_{[J_0],[J_1], [Y_0],[Y_1]\atop [I_0],[I_1],[M_0],[M_1]}v^yq^z\frac{a_{B_0}a_{C_0}a_{J_0}a_{J_1}}{a_{Y_0}}F_{B_0,J_1[-1],J_0[1],C_0}^{Y_0}
\frac{a_{B_1}a_{C_1}a_{J_0}a_{J_1}}{a_{Y_1}}F_{B_1,J_0[-1],J_1[1],C_1}^{Y_1}\frac{1}{a_{J_0}a_{J_1}}\\
&\quad\quad\quad\quad\frac{a_{A_0}a_{Y_0}a_{I_0}a_{I_1}}{a_{M_0}}F_{A_0,I_1[-1],I_0[1],Y_0}^{M_0}
\frac{a_{A_1}a_{Y_1}a_{I_0}a_{I_1}}{a_{M_1}}F_{A_1,I_0[-1],I_1[1],Y_1}^{M_1}\frac{1}{a_{I_0}a_{I_1}}\\&
\quad\quad\quad\quad\quad\quad\quad\quad K_{\hat{I}_0+\hat{J}_0}K_{\hat{I}_1+\hat{J}_1}^\ast u_{M_0\oplus M_1[1]},\end{flalign*}
where $z=\lr{\hat{J}_1,\hat{J}_0}-\lr{\hat{B}_0,\hat{J}_0}-\lr{\hat{J}_1,\hat{C}_0}+\lr{\hat{J}_0,\hat{J}_1}-\lr{\hat{B}_1,\hat{J}_1}-\lr{\hat{J}_0,\hat{C}_1}
+\lr{\hat{I}_1,\hat{I}_0}-\lr{\hat{A}_0,\hat{I}_0}-\lr{\hat{I}_1,\hat{Y}_0}+\lr{\hat{I}_0,\hat{I}_1}-\lr{\hat{A}_1,\hat{I}_1}-\lr{\hat{I}_0,\hat{Y}_1}$.

Using Corollary \ref{tuilun1}, we obtain that
\begin{flalign*}&u_{A_0\oplus A_1[1]}(u_{B_0\oplus B_1[1]}u_{C_0\oplus C_1[1]})=\\&
a_{A_0}a_{B_0}a_{C_0}a_{A_1}a_{B_1}a_{C_1}\sum\limits_{[I_0],[I_1], [Y_0],[Y_1]\atop [J_0],[J_1],[M_0],[M_1]}v^yq^z\frac{a_{I_0}a_{I_1}a_{J_0}a_{J_1}}{a_{M_0}a_{M_1}}\\&
(\sum\limits_{[N_0],[L_0]}\frac{a_{N_0}a_{L_0}}{a_{A_0}a_{Y_0}}F_{I_1,N_0}^{A_0}F_{L_0,I_0}^{Y_0}F_{N_0,L_0}^{M_0})
(\sum\limits_{[N_1],[L_1]}\frac{a_{N_1}a_{L_1}}{a_{A_1}a_{Y_1}}F_{I_0,N_1}^{A_1}F_{L_1,I_1}^{Y_1}F_{N_1,L_1}^{M_1})\\&
(\sum\limits_{[S_0],[T_0]}\frac{a_{S_0}a_{T_0}}{a_{B_0}a_{C_0}}F_{J_1,S_0}^{B_0}F_{T_0,J_0}^{C_0}F_{S_0,T_0}^{Y_0})
(\sum\limits_{[S_1],[T_1]}\frac{a_{S_1}a_{T_1}}{a_{B_1}a_{C_1}}F_{J_0,S_1}^{B_1}F_{T_1,J_1}^{C_1}F_{S_1,T_1}^{Y_1})
\\&\quad\quad\quad\quad K_{\hat{I}_0+\hat{J}_0}K_{\hat{I}_1+\hat{J}_1}^\ast u_{M_0\oplus M_1[1]}.\end{flalign*}

Noting that $\hat{Y}_0=\hat{B}_0+\hat{C}_0-\hat{J}_0-\hat{J}_1$ and $\hat{Y}_1=\hat{B}_1+\hat{C}_1-\hat{J}_0-\hat{J}_1$, we have that
\begin{flalign*}&u_{A_0\oplus A_1[1]}(u_{B_0\oplus B_1[1]}u_{C_0\oplus C_1[1]})=\\&
\sum\limits_{[I_0],[I_1],[J_0],[J_1],[M_0], [M_1], [N_0]\atop [L_0],[N_1],[L_1],[S_0],[T_0],[S_1],[T_1]}v^yq^z\frac{a_{I_0}a_{I_1}a_{J_0}a_{J_1}a_{N_0}a_{L_0}a_{N_1}a_{L_1}a_{S_0}a_{T_0}a_{S_1}a_{T_1}}{a_{M_0}a_{M_1}}\\&
F_{I_1,N_0}^{A_0}F_{J_1,S_0}^{B_0}F_{T_0,J_0}^{C_0}F_{I_0,N_1}^{A_1}F_{J_0,S_1}^{B_1}F_{T_1,J_1}^{C_1}F_{N_0,L_0}^{M_0}F_{N_1,L_1}^{M_1}\sum\limits_{[Y_0]}F_{L_0,I_0}^{Y_0}F_{S_0,T_0}^{Y_0}\frac{1}{a_{Y_0}}\\&
\sum\limits_{[Y_1]}F_{L_1,I_1}^{Y_1}F_{S_1,T_1}^{Y_1}\frac{1}{a_{Y_1}}
K_{\hat{I}_0+\hat{J}_0}K_{\hat{I}_1+\hat{J}_1}^\ast u_{M_0\oplus M_1[1]}.\end{flalign*}

Using Green's formula, we have that
\begin{flalign*}
\sum\limits_{[Y_0]}F_{L_{0},I_{0}}^{Y_0}F_{S_{0},T_0}^{Y_0}\frac{1}{a_{Y_0}}=\sum\limits_{[K_0],[P_0],[Q_0],[R_0]}
q^{-\lr{\hat{K}_0,\hat{R}_0}}F_{K_0,P_0}^{L_0}F_{Q_0,R_0}^{I_0}F_{K_0,Q_0}^{S_{0}}F_{P_0,R_0}^{T_0}\frac{a_{K_0}a_{P_0}a_{Q_0}a_{R_0}}{a_{L_0}a_{I_0}a_{S_{0}}a_{T_0}}
\end{flalign*}
and
\begin{flalign*}
\sum\limits_{[Y_1]}F_{L_{1},I_{1}}^{Y_1}F_{S_{1},T_1}^{Y_1}\frac{1}{a_{Y_1}}=\sum\limits_{[K_1],[P_1],[Q_1],[R_1]}
q^{-\lr{\hat{K}_1,\hat{R}_1}}F_{K_1,P_1}^{L_1}F_{Q_1,R_1}^{I_1}F_{K_1,Q_1}^{S_{1}}F_{P_1,R_1}^{T_1}\frac{a_{K_1}a_{P_1}a_{Q_1}a_{R_1}}{a_{L_1}a_{I_1}a_{S_{1}}a_{T_1}}.
\end{flalign*}

Hence, we obtain that
\begin{flalign*}&u_{A_0\oplus A_1[1]}(u_{B_0\oplus B_1[1]}u_{C_0\oplus C_1[1]})=\\&
\sum\limits_{[I_0],[I_1],[J_0],[J_1],[M_0], [M_1], [N_0], [L_0],[N_1],[L_1],[S_0]\atop [T_0],[S_1],[T_1],[K_0],[P_0],[Q_0],[R_0],[K_1],[P_1],[Q_1],[R_1]}v^yq^{z'}\frac{a_{J_0}a_{J_1}a_{N_0}a_{N_1}a_{K_0}a_{P_0}a_{Q_0}a_{R_0}a_{K_1}a_{P_1}a_{Q_1}a_{R_1}}{a_{M_0}a_{M_1}}\\&
\quad\quad\quad\quad\quad  F_{I_1,N_0}^{A_0}F_{J_1,S_0}^{B_0}F_{T_0,J_0}^{C_0}F_{I_0,N_1}^{A_1}F_{J_0,S_1}^{B_1}
F_{T_1,J_1}^{C_1}F_{N_0,L_0}^{M_0}F_{N_1,L_1}^{M_1}F_{K_0,P_0}^{L_0}F_{Q_0,R_0}^{I_0}\\&
\quad\quad\quad\quad\quad  F_{K_0,Q_0}^{S_{0}}F_{P_0,R_0}^{T_0}F_{K_1,P_1}^{L_1}F_{Q_1,R_1}^{I_1}F_{K_1,Q_1}^{S_{1}}F_{P_1,R_1}^{T_1}
K_{\hat{I}_0+\hat{J}_0}K_{\hat{I}_1+\hat{J}_1}^\ast u_{M_0\oplus M_1[1]},\end{flalign*}
where $z'=z-\lr{\hat{K}_0,\hat{R}_0}-\lr{\hat{K}_1,\hat{R}_1}$.

Noting that $\hat{I}_0=\hat{Q}_0+\hat{R}_0$, $\hat{I}_1=\hat{Q}_1+\hat{R}_1$ and
applying the associativity formula (\ref{associativity}),
we get that
\begin{equation}\label{youbian}
\begin{split}&u_{A_0\oplus A_1[1]}(u_{B_0\oplus B_1[1]}u_{C_0\oplus C_1[1]})=\\&
\sum\limits_{[J_0],[J_1],[M_0], [M_1], [N_0],[N_1],[K_0]\atop[P_0],[Q_0],[R_0],[K_1],[P_1],[Q_1],[R_1]}v^yq^{z'}\frac{a_{J_0}a_{J_1}a_{N_0}a_{N_1}a_{K_0}a_{P_0}a_{Q_0}a_{R_0}a_{K_1}a_{P_1}a_{Q_1}a_{R_1}}{a_{M_0}a_{M_1}}\\&
F_{Q_1,R_1,N_0}^{A_0}F_{J_1,K_0,Q_0}^{B_0}F_{P_0,R_0,J_0}^{C_0}F_{Q_0,R_0,N_1}^{A_1}F_{J_0,K_1,Q_1}^{B_1}F_{P_1,R_1,J_1}^{C_1}F_{N_0,K_0,P_0}^{M_0}F_{N_1,K_1,P_1}^{M_1}\\&
\quad\quad\quad\quad\quad\quad\quad\quad\quad\quad K_{\hat{Q}_0+\hat{R}_0+\hat{J}_0}K_{\hat{Q}_1+\hat{R}_1+\hat{J}_1}^\ast u_{M_0\oplus M_1[1]}.\end{split}\end{equation}

Replacing the notations $I_0,I_1,T_0,T_1,U_0,V_0,W_0,Z_0,U_1,V_1,W_1,Z_1$ in (\ref{zuobian}) by $Q_0,Q_1,P_0$, $P_1,R_1,N_0,J_1,K_0,R_0,N_1,J_0,K_1$, respectively,
we conclude that all terms in (\ref{zuobian}) are the same as those in (\ref{youbian}) except the exponents of $v$ and $q$.

Now let us compare the exponents of $v$ and $q$ in (\ref{zuobian}) and (\ref{youbian}). Replacing the notations in (\ref{zuobian}) as above, we have that
\begin{equation}\label{lcommon2}
\begin{split}
b=&\lr{\hat{Q}_1-\hat{Q}_0,\hat{A}_0-\hat{A}_1+\hat{B}_0-\hat{B}_1}+\lr{\hat{A}_0,\hat{B}_0}+\lr{\hat{A}_1,\hat{B}_1}\\&+\lr{\hat{R}_1+\hat{J}_1-\hat{R}_0-\hat{J}_0,\hat{A}_0-\hat{A}_1+\hat{B}_0-\hat{B}_1+\hat{C}_0-\hat{C}_1}\\
&+\lr{\hat{A}_0+\hat{B}_0-\hat{Q}_0-\hat{Q}_1,\hat{C}_0}+\lr{\hat{A}_1+\hat{B}_1-\hat{Q}_0-\hat{Q}_1,\hat{C}_1}
\end{split}\end{equation}
and
\begin{equation}\label{lcommon3}
\begin{split}
c'=&\lr{\hat{Q}_1,\hat{Q}_0}-\lr{\hat{A}_0,\hat{Q}_0}-\lr{\hat{Q}_1,\hat{B}_0}+\lr{\hat{Q}_0,\hat{Q}_1}-\lr{\hat{A}_1,\hat{Q}_1}-\lr{\hat{Q}_0,\hat{B}_1}\\
&+\lr{\hat{R}_1+\hat{J}_1,\hat{R}_0+\hat{J}_0}-\lr{\hat{A}_0+\hat{B}_0-\hat{Q}_0-\hat{Q}_1,\hat{R}_0+\hat{J}_0}-\lr{\hat{R}_1+\hat{J}_1,\hat{C}_0}\\
&+\lr{\hat{R}_0+\hat{J}_0,\hat{R}_1+\hat{J}_1}-\lr{\hat{A}_1+\hat{B}_1-\hat{Q}_0-\hat{Q}_1,\hat{R}_1+\hat{J}_1}-\lr{\hat{R}_0+\hat{J}_0,\hat{C}_1}\\
&-\lr{\hat{R}_1,\hat{K}_0}-\lr{\hat{R}_0,\hat{K}_1}.
\end{split}\end{equation}
While,
\begin{equation}\label{rcommon2}
\begin{split}
y=&\lr{\hat{J}_1-\hat{J}_0,\hat{B}_0-\hat{B}_1+\hat{C}_0-\hat{C}_1}+\lr{\hat{B}_0,\hat{C}_0}+\lr{\hat{B}_1,\hat{C}_1}+(\hat{J}_1-\hat{J}_0,\hat{A}_0-\hat{A}_1)\\
&+\lr{\hat{Q}_1+\hat{R}_1-\hat{Q}_0-\hat{R}_0,\hat{A}_0-\hat{A}_1+\hat{B}_0-\hat{B}_1+\hat{C}_0-\hat{C}_1}\\
&+\lr{\hat{A}_0,\hat{B}_0+\hat{C}_0-\hat{J}_0-\hat{J}_1}+\lr{\hat{A}_1,\hat{B}_1+\hat{C}_1-\hat{J}_0-\hat{J}_1}
\end{split}\end{equation}
and
\begin{equation}\label{rcommon3}
\begin{split}
z'=&\lr{\hat{J}_1,\hat{J}_0}-\lr{\hat{B}_0,\hat{J}_0}-\lr{\hat{J}_1,\hat{C}_0}+\lr{\hat{J}_0,\hat{J}_1}-\lr{\hat{B}_1,\hat{J}_1}-\lr{\hat{J}_0,\hat{C}_1}\\
&+\lr{\hat{Q}_1+\hat{R}_1,\hat{Q}_0+\hat{R}_0}-\lr{\hat{A}_0,\hat{Q}_0+\hat{R}_0}-\lr{\hat{Q}_1+\hat{R}_1,\hat{B}_0+\hat{C}_0-\hat{J}_0-\hat{J}_1}\\
&+\lr{\hat{Q}_0+\hat{R}_0,\hat{Q}_1+\hat{R}_1}-\lr{\hat{A}_1,\hat{Q}_1+\hat{R}_1}-\lr{\hat{Q}_0+\hat{R}_0,\hat{B}_1+\hat{C}_1-\hat{J}_0-\hat{J}_1}\\
&-\lr{\hat{K}_0,\hat{R}_0}-\lr{\hat{K}_1,\hat{R}_1}.
\end{split}\end{equation}

By direct calculations, we get that
\begin{equation*}
b=-\lr{\hat{Q}_0,\hat{C}_1}-\lr{\hat{Q}_1,\hat{C}_0}+\text{Common~terms}~(\spadesuit)
\end{equation*}
and
\begin{equation*}
y=\lr{\hat{Q}_0,\hat{C}_1}+\lr{\hat{Q}_1,\hat{C}_0}-2\lr{\hat{A}_0,\hat{J}_0}-2\lr{\hat{A}_1,\hat{J}_1}+\text{Common~terms}~(\spadesuit),
\end{equation*}
where $(\spadesuit)$ denotes the common terms of $(\ref{lcommon2})$ and $(\ref{rcommon2})$.

Noting that $\hat{K}_0=\hat{B}_0-\hat{Q}_0-\hat{J}_1$ and $\hat{K}_1=\hat{B}_1-\hat{Q}_1-\hat{J}_0$, we obtain that
\begin{equation*}
c'=-\lr{\hat{A}_0,\hat{J}_0}-\lr{\hat{A}_1,\hat{J}_1}+\text{Common~terms}~(\clubsuit)
\end{equation*}
and
\begin{equation*}
z'=-\lr{\hat{Q}_1,\hat{C}_0}-\lr{\hat{Q}_0,\hat{C}_1}+\text{Common~terms}~(\clubsuit),
\end{equation*}
where $(\clubsuit)$ denotes the common terms of $(\ref{lcommon3})$ and $(\ref{rcommon3})$.

Therefore, we have that $b+2c'=y+2z'$ and complete the proof.
\end{proof}

\section{Realizations of Drinfeld double Hall algebras}
In this section, we show that the $2$-periodic extended derived Hall algebra $\mathcal {D}\mathcal {H}_2(\A)$ is isomorphic to the Drinfeld double Hall algebra $\mathcal {D}(\A)$.

By the definition of $\mathcal {D}\mathcal {H}_2(\A)$, we immediately get the following embeddings.
\begin{lemma}\label{embeds}
There exist two embeddings of algebras
$$i^+:{\mathcal{H}}_{\rm{tw}}^{\rm e}(\mathcal{A})\hookrightarrow \mathcal {D}\mathcal {H}_2(\A), u_{[M]}K_\alpha\mapsto u_{M}K_\alpha$$
and $$i^-:{\mathcal{H}}_{\rm{tw}}^{\rm e}(\mathcal{A})\hookrightarrow \mathcal {D}\mathcal {H}_2(\A), u_{[M]}K_\alpha\mapsto u_{M[1]}K_\alpha^\ast.$$
\end{lemma}
\begin{proof}
By Remark \ref{dyjx}, $K_{\alpha_0}K_{\beta_0}=K_{\alpha_0+\beta_0}$ and $K_{\alpha_1}^\ast K_{\beta_1}^\ast =K_{\alpha_1+\beta_1}^\ast$ for any $\alpha_0,\beta_0,\alpha_1,\beta_1\in K(\A)$. For any $\beta_0\in K(\A)$ and $[A_0],[A_1]\in\Iso(\A)$, we have that
$$u_{A_0}K_{\beta_0}=v^{-(\beta_0,\hat{A}_0)}K_{\beta_0}u_{A_0}~\text{and}~u_{A_1[1]}K_{\beta_0}=v^{(\beta_0,\hat{A}_1)}K_{\beta_0}u_{A_1[1]}.$$

For any $[A_0],[B_0],[A_1],[B_1]\in\Iso(\A)$, by Definition \ref{maindef} together with the exact sequence (\ref{xyczhl}), we have that
$$u_{A_0}u_{B_0}=\sum\limits_{[M_0]\in{\rm Iso}(\A)}v^{\lr{\hat{A}_0,\hat{B}_0}}H_{A_0,B_0}^{M_0}u_{M_0}~~\text{and}~~u_{A_1[1]}u_{B_1[1]}=\sum\limits_{[M_1]\in{\rm Iso}(\A)}v^{\lr{\hat{A}_1,\hat{B}_1}}H_{A_1,B_1}^{M_1}u_{M_1[1]}.$$
By Corollary \ref{tuilun2.7}, for any $L,M,N\in\A$ we have that $$H_{M,N}^L=\frac{a_Ma_N}{a_L}F_{M,N}^L=\frac{a_Ma_N}{a_L}g_{M,N}^L=\frac{{|\Ext_\mathcal{A}^1{{(M,N)}_L}|}}{{|\Hom_\mathcal{A}(M,N)|}}.$$

Taking all these together, we obtain that $i^+,i^-$ are homomorphisms of algebras. Since $i^+,i^-$ send a basis to a linearly independent set in $\mathcal {D}\mathcal {H}_2(\A)$, we conclude that they are injective.
\end{proof}

In what follows, we provide another basis for the $2$-periodic extended derived Hall algebra $\mathcal {D}\mathcal {H}_2(\A)$, called the {\em triangular basis}.
\begin{proposition}\label{sanjiaoji}
The $2$-periodic extended derived Hall algebra $\mathcal {D}\mathcal {H}_2(\A)$ has a basis given by
$$\{K_\alpha K_\beta^\ast u_{A_0}u_{B_1[1]}~|~\alpha,\beta\in K(\A),[A_0],[B_1]\in\Iso(\A)\}.$$
\end{proposition}
\begin{proof}
The following proof is inspired by the proof of the PBW-basis of the Hall algebra of a module category, which is given in \cite[Theorem 3.1]{GuoP}.

For any $\alpha,\beta\in K(\A),[A_0],[B_1]\in\Iso(\A)$, define \begin{flalign*}\delta(K_\alpha K_\beta^\ast u_{A_0\oplus B_1[1]})&=\dim_k\Ext_{\mathcal {R}(\A)}^1(\widetilde{A}_0\oplus \widetilde{B}_1[1],\widetilde{A}_0\oplus \widetilde{B}_1[1])\\&=\dim_k\Ext_{\mathcal {R}(\A)}^1(\widetilde{A}_0[1]\oplus \widetilde{B}_1,\widetilde{A}_0[1]\oplus \widetilde{B}_1).\end{flalign*}

If $\delta(K_\alpha K_\beta^\ast u_{A_0\oplus B_1[1]})=0$, by (\ref{tonggou}), we obtain that $\Hom_{\A}(A_0,B_1)=0$. Thus, $$K_\alpha K_\beta^\ast u_{A_0}u_{B_1[1]}=K_\alpha K_\beta^\ast u_{A_0\oplus B_1[1]}.$$

If $\delta(K_\alpha K_\beta^\ast u_{A_0\oplus B_1[1]})>0$, we have that
\begin{equation}\label{basis}
\begin{split}
&K_\alpha K_\beta^\ast u_{A_0}u_{B_1[1]}=\\&K_\alpha K_\beta^\ast u_{A_0\oplus B_1[1]}+\sum\limits_{[I_1]\neq[0],[M_0],[M_1]}v^{\lr{\hat{I}_1,\hat{A}_0-\hat{B}_1}}\frac{H_{A_0,I_1[-1]}^{M_0}H_{I_1[1],B_1}^{M_1}}{a_{I_1}}K_\alpha K_{\beta+\hat{I}_1}^\ast u_{M_0\oplus M_1[1]}.\end{split}
\end{equation}
For any $[I_1],[M_0],[M_1]\in\Iso(\A)$, if $H_{A_0,I_1[-1]}^{M_0}H_{I_1[1],B_1}^{M_1}\neq0$, then we have an exact sequence in $\A$
$$\xymatrix@R=0.8pc{0\ar[r]&M_0\ar[r]&A_0\ar[rr]^-{f}\ar@{->>}[rd]&&B_1\ar[r]&M_1\ar[r]&0.\\
&&&I_1\ar@{>->}[ru]&&&}$$
Thus, we obtain a triangle in $D^b(\A)$
$$\xymatrix{A_0\ar[r]^-f&B_1\ar[r]&M_0[1]\oplus M_1\ar[r]&A_0[1],}$$ which gives a triangle in $\mathcal {R}(\A)$
\begin{equation}\label{kelie}
\xymatrix{\widetilde{A}_0\ar[r]^-{\widetilde{f}}&\widetilde{B}_1\ar[r]&\widetilde{M}_0[1]\oplus \widetilde{M}_1\ar[r]&\widetilde{A}_0[1].}\end{equation}
By \cite[Lemma 4.2]{RuanZ}, we have that $$\dim_k\Ext_{\mathcal {R}(\A)}^1(\widetilde{M}_0[1]\oplus \widetilde{M}_1,\widetilde{M}_0[1]\oplus \widetilde{M}_1)\leq\dim_k\Ext_{\mathcal {R}(\A)}^1(\widetilde{A}_0[1]\oplus \widetilde{B}_1,\widetilde{A}_0[1]\oplus \widetilde{B}_1)$$
and the equality holds if and only if the triangle (\ref{kelie}) splits if and only if $\widetilde{f}=0$, which is also equivalent to $f=0$, since $\A$ is fully embedded into $\mathcal {R}(\A)$.

Hence, for each term $K_\alpha K_{\beta+\hat{I}_1}^\ast u_{M_0\oplus M_1[1]}$ with $[I_1]\neq[0]$ in (\ref{basis}), we have that $$\delta(K_\alpha K_{\beta+\hat{I}_1}^\ast u_{M_0\oplus M_1[1]})<\delta(K_\alpha K_\beta^\ast u_{A_0\oplus B_1[1]}).$$ Since $\{K_\alpha K_\beta^\ast u_{A_0\oplus B_1[1]}~|~\alpha,\beta\in K(\A),[A_0],[B_1]\in\Iso(\A)\}$ is a basis of $\mathcal {D}\mathcal {H}_2(\A)$, according to (\ref{basis}), we obtain that $\mathfrak{B}:=\{K_\alpha K_\beta^\ast u_{A_0}u_{B_1[1]}~|~\alpha,\beta\in K(\A),[A_0],[B_1]\in\Iso(\A)\}$ is a linear independent set. By induction on $\delta(K_\alpha K_\beta^\ast u_{A_0\oplus B_1[1]})$ for all $\alpha,\beta\in K(\A),[A_0],[B_1]\in\Iso(\A)$, we get that $K_\alpha K_\beta^\ast u_{A_0\oplus B_1[1]}$ can be written as a linear combination of elements in $\mathfrak{B}$. Therefore, $\mathfrak{B}$ is a basis of $\mathcal {D}\mathcal {H}_2(\A)$.
\end{proof}
\begin{corollary}\label{twocopy}
The multiplication map $\mu: a\otimes b\mapsto i^+(a)i^-(b)$ defines an isomorphism of vector spaces
$$\mu:{\mathcal{H}}_{\rm{tw}}^{\rm e}(\mathcal{A})\otimes {\mathcal{H}}_{\rm{tw}}^{\rm e}(\mathcal{A})\longrightarrow\mathcal {D}\mathcal {H}_2(\A).$$
\end{corollary}
\begin{theorem}\label{tgdly}
The $2$-periodic extended derived Hall algebra $\mathcal {D}\mathcal {H}_2(\A)$ is isomorphic to the Drinfeld double Hall algebra $\mathcal {D}(\A)$.
\end{theorem}
\begin{proof}
By Lemma \ref{embeds} and Corollary \ref{twocopy},
we only need to prove that the commutator relation
\begin{equation}\label{Drinfeld2}
\sum\varphi(a_{(2)},b_{(1)})i^+(a_{(1)}) i^-(b_{(2)})=\sum\varphi(a_{(1)},b_{(2)})i^-(b_{(1)}) i^+(a_{(2)})\end{equation} holds in $\mathcal {D}\mathcal {H}_2(\mathcal{A})$ for each $a=u_{[X_0]}K_{\alpha}$ and $b=u_{[X_1]}K_{\beta}$, where $\alpha,\beta\in K(\A)$ and $[X_0],[X_1]\in\Iso(\A)$.

Firstly, we prove the case that $a=u_{[X_0]}$ and $b=u_{[X_1]}$.
Since $$\Delta(u_{X_0})=\sum\limits_{[A_0],[B_0]}v^{\lr{\hat{A}_0,\hat{B}_0}}F_{A_0,B_0}^{X_0}u_{A_0}K_{\hat{B}_0}\otimes u_{B_0}$$ and
$$\Delta(u_{X_1})=\sum\limits_{[A_1],[B_1]}v^{\lr{\hat{A}_1,\hat{B}_1}}F_{A_1,B_1}^{X_1}u_{A_1}K_{\hat{B}_1}\otimes u_{B_1},$$
the left hand side of $(\ref{Drinfeld2})$ becomes
\begin{equation*}\begin{split}
\mbox{LHS~~of}~~(\ref{Drinfeld2})&=\sum\limits_{[A_0],[B_0],[A_1],[B_1]}v^{\lr{\hat{A}_0,\hat{B}_0}+\lr{\hat{A}_1,\hat{B}_1}}F_{A_0,B_0}^{X_0}
F_{A_1,B_1}^{X_1}\varphi(u_{B_0},u_{A_1}K_{\hat{B}_1})u_{A_0}K_{\hat{B}_0}u_{{B_1}[1]}\\
&=\sum\limits_{[A_0],[B_0],[B_1]}v^{\lr{\hat{A}_0,\hat{B}_0}+\lr{\hat{B}_0,\hat{B}_1}}F_{A_0,B_0}^{X_0}
F_{B_0,B_1}^{X_1}a_{B_0}u_{A_0}K_{\hat{B}_0}u_{{B_1}[1]}\\
&=\sum\limits_{[A_0],[B_0],[B_1]}v^{\lr{\hat{B}_0,\hat{B}_1}-\lr{\hat{B}_0,\hat{A}_0}}F_{A_0,B_0}^{X_0}
F_{B_0,B_1}^{X_1}a_{B_0}K_{\hat{B}_0}u_{A_0}u_{{B_1}[1]}.
\end{split}\end{equation*}
By Definition \ref{maindef}, we obtain that
\begin{equation*}\begin{split}
&\mbox{LHS~~of}~~(\ref{Drinfeld2})=\\&\sum\limits_{[A_0],[B_0],[B_1]\atop [I_1],[M_0],[M_1]}v^{\lr{\hat{B}_0,\hat{B}_1}-\lr{\hat{B}_0,\hat{A}_0}+\lr{\hat{I}_1,\hat{A}_0-\hat{B}_1}}F_{A_0,B_0}^{X_0}
F_{B_0,B_1}^{X_1}a_{B_0}\frac{H_{A_0,I_1[-1]}^{M_0}H_{I_1[1],B_1}^{M_1}}{a_{I_1}}K_{\hat{B}_0}K_{\hat{I}_1}^\ast u_{M_0\oplus M_1[1]}.
\end{split}\end{equation*}
By Corollary \ref{tuilun2.7} and Lemma \ref{diyiy}, we have that $$H_{A_0,I_1[-1]}^{M_0}=a_{I_1}F_{I_1,M_0}^{A_0}~\text{and}~H_{I_1[1],B_1}^{M_1}=a_{I_1}F_{M_1,I_1}^{B_1}.$$
Thus, we obtain that
\begin{equation*}\begin{split}
&\mbox{LHS~~of}~~(\ref{Drinfeld2})=\\&\sum\limits_{[A_0],[B_0],[B_1]\atop [I_1],[M_0],[M_1]}v^{\lr{\hat{B}_0,\hat{B}_1}-\lr{\hat{B}_0,\hat{A}_0}+\lr{\hat{I}_1,\hat{A}_0-\hat{B}_1}}F_{I_1,M_0}^{A_0}F_{A_0,B_0}^{X_0}
F_{B_0,B_1}^{X_1}F_{M_1,I_1}^{B_1}a_{B_0}a_{I_1}K_{\hat{B}_0}K_{\hat{I}_1}^\ast u_{M_0\oplus M_1[1]}.
\end{split}\end{equation*}
Applying the associativity formula (\ref{associativity}) together with $\hat{A}_0=\hat{M}_0+\hat{I}_1$ and $\hat{B}_1=\hat{M}_1+\hat{I}_1$, we obtain that
\begin{equation*}
\begin{split}
&\mbox{LHS~~of}~~(\ref{Drinfeld2})=\\&\sum\limits_{[B_0],[I_1],[M_0],[M_1]}v^{\lr{\hat{B}_0,\hat{M}_1+\hat{I}_1}-\lr{\hat{B}_0,\hat{M}_0+\hat{I}_1}+\lr{\hat{I}_1,\hat{M}_0-\hat{M}_1}}F_{I_1,M_0,B_0}^{X_0}
F_{B_0,M_1,I_1}^{X_1}a_{B_0}a_{I_1}K_{\hat{B}_0}K_{\hat{I}_1}^\ast u_{M_0\oplus M_1[1]}\end{split}\end{equation*}
\begin{equation}\label{4.3zuo}
\begin{split}=\sum\limits_{[B_0],[I_1],[M_0],[M_1]}v^{\lr{\hat{I}_1-\hat{B}_0,\hat{M}_0-\hat{M}_1}}F_{I_1,M_0,B_0}^{X_0}
F_{B_0,M_1,I_1}^{X_1}a_{B_0}a_{I_1}K_{\hat{B}_0}K_{\hat{I}_1}^\ast u_{M_0\oplus M_1[1]}.
\end{split}\end{equation}

On the other hand, the right hand side of $(\ref{Drinfeld2})$ becomes
\begin{equation*}\begin{split}
\mbox{RHS~~of}~~(\ref{Drinfeld2})&=\sum\limits_{[A_0],[B_0],[A_1],[B_1]}v^{\lr{\hat{A}_0,\hat{B}_0}+\lr{\hat{A}_1,\hat{B}_1}}F_{A_0,B_0}^{X_0}
F_{A_1,B_1}^{X_1}\varphi(u_{A_0}K_{\hat{B}_0},u_{B_1})u_{A_1[1]}K_{\hat{B}_1}^\ast u_{{B_0}}\\
&=\sum\limits_{[A_0],[B_0],[A_1]}v^{\lr{\hat{A}_0,\hat{B}_0}+\lr{\hat{A}_1,\hat{A}_0}-(\hat{A}_1,\hat{A}_0)}F_{A_0,B_0}^{X_0}
F_{A_1,A_0}^{X_1}a_{A_0}K_{\hat{A}_0}^\ast u_{A_1[1]}u_{{B_0}}.
\end{split}\end{equation*}
By Definition \ref{maindef}, we obtain that
\begin{equation*}\begin{split}
&\mbox{RHS~~of}~~(\ref{Drinfeld2})=\\&=\sum\limits_{[A_0],[B_0],[A_1]\atop [I_0],[M_0],[M_1]}v^{\lr{\hat{A}_0,\hat{B}_0}-\lr{\hat{A}_0,\hat{A}_1}+\lr{\hat{I}_0,\hat{A}_1-\hat{B}_0}}F_{A_0,B_0}^{X_0}
F_{A_1,A_0}^{X_1}a_{A_0}\frac{H_{I_0[1],B_0}^{M_0}H_{A_1,I_0[-1]}^{M_1}}{a_{I_0}}K_{\hat{A}_0}^\ast K_{\hat{I}_0} u_{M_0\oplus M_1[1]}.
\end{split}\end{equation*}
By Corollary \ref{tuilun2.7} and Lemma \ref{diyiy}, we have that $$H_{I_0[1],B_0}^{M_0}=a_{I_0}F_{M_0,I_0}^{B_0}~\text{and}~H_{A_1,I_0[-1]}^{M_1}=a_{I_0}F_{I_0,M_1}^{A_1}.$$
Thus, we get that
\begin{equation*}\begin{split}
&\mbox{RHS~~of}~~(\ref{Drinfeld2})=\\&\sum\limits_{[A_0],[B_0],[A_1]\atop [I_0],[M_0],[M_1]}v^{\lr{\hat{I}_0-\hat{A}_0,\hat{A}_1-\hat{B}_0}}F_{A_0,B_0}^{X_0}F_{M_0,I_0}^{B_0}
F_{I_0,M_1}^{A_1}F_{A_1,A_0}^{X_1}a_{A_0}a_{I_0}K_{\hat{I}_0}K_{\hat{A}_0}^\ast u_{M_0\oplus M_1[1]}.
\end{split}\end{equation*}
Applying the associativity formula (\ref{associativity}) together with $\hat{A}_1=\hat{M}_1+\hat{I}_0$ and $\hat{B}_0=\hat{M}_0+\hat{I}_0$, we obtain that
\begin{equation}\label{4.3you}
\begin{split}
\mbox{RHS~~of}~~(\ref{Drinfeld2})=\sum\limits_{[A_0],[I_0],[M_0],[M_1]}v^{\lr{\hat{A}_0-\hat{I}_0,\hat{M}_0-\hat{M}_1}}F_{A_0,M_0,I_0}^{X_0}
F_{I_0,M_1,A_0}^{X_1}a_{A_0}a_{I_0}K_{\hat{I}_0}K_{\hat{A}_0}^\ast u_{M_0\oplus M_1[1]}.
\end{split}\end{equation}
Replacing the notations $[I_1]$ and $[B_0]$ in $(\ref{4.3zuo})$ by $[A_0]$ and $[I_0]$ in $(\ref{4.3you})$, respectively, we obtain that
$\mbox{LHS~~of}~~(\ref{Drinfeld2})=\mbox{RHS~~of}~~(\ref{Drinfeld2})$ in this case.

Now we prove the general case that $a=u_{[X_0]}K_\alpha$ and $b=u_{[X_1]}K_\beta$.
At this point, the left hand side of $(\ref{Drinfeld2})$ becomes
\begin{equation*}\begin{split}
&\mbox{LHS~~of}~~(\ref{Drinfeld2})=\\&\sum\limits_{[A_0],[B_0],[A_1],[B_1]}v^{\lr{\hat{A}_0,\hat{B}_0}+\lr{\hat{A}_1,\hat{B}_1}}F_{A_0,B_0}^{X_0}
F_{A_1,B_1}^{X_1}\varphi(u_{B_0}K_\alpha,u_{A_1}K_{\hat{B}_1+\beta})u_{A_0}K_{\hat{B}_0+\alpha}u_{{B_1}[1]}K_\beta^*\\
&=v^{(\alpha,\beta)}\sum\limits_{[A_0],[B_0],[A_1],[B_1]}v^{\lr{\hat{A}_0,\hat{B}_0}+\lr{\hat{A}_1,\hat{B}_1}}F_{A_0,B_0}^{X_0}
F_{A_1,B_1}^{X_1}\varphi(u_{B_0},u_{A_1}K_{\hat{B}_1})u_{A_0}K_{\hat{B}_0}u_{{B_1}[1]}K_\alpha K_\beta^*
\end{split}\end{equation*}
and the right hand side of $(\ref{Drinfeld2})$ becomes
\begin{equation*}\begin{split}
&\mbox{RHS~~of}~~(\ref{Drinfeld2})=\\&\sum\limits_{[A_0],[B_0],[A_1],[B_1]}v^{\lr{\hat{A}_0,\hat{B}_0}+\lr{\hat{A}_1,\hat{B}_1}}F_{A_0,B_0}^{X_0}
F_{A_1,B_1}^{X_1}\varphi(u_{A_0}K_{\hat{B}_0+\alpha},u_{B_1}K_\beta)u_{A_1[1]}K_{\hat{B}_1+\beta}^\ast u_{{B_0}}K_\alpha\\
&=v^{(\alpha,\beta)}\sum\limits_{[A_0],[B_0],[A_1],[B_1]}v^{\lr{\hat{A}_0,\hat{B}_0}+\lr{\hat{A}_1,\hat{B}_1}}F_{A_0,B_0}^{X_0}
F_{A_1,B_1}^{X_1}\varphi(u_{A_0}K_{\hat{B}_0},u_{B_1})u_{A_1[1]}K_{\hat{B}_1}^\ast u_{{B_0}}K_\alpha K_\beta^*.
\end{split}\end{equation*}
Thus, in this case, $\mbox{LHS~~of}~~(\ref{Drinfeld2})=\mbox{RHS~~of}~~(\ref{Drinfeld2})$ is the same as the first case as long as we eliminate $v^{(\alpha,\beta)}K_\alpha K_\beta^*$ on both sides. Therefore, we complete the proof.
\end{proof}
\begin{remark}
$(1)$~In Theorem \ref{tgdly},
if ${\mathcal{H}}_{\rm{tw}}^{\rm e}(\mathcal{A})$ is a topological bialgebra, we need to complete the Hall algebra $\mathcal {D}\mathcal {H}_2(\A)$
with respect to the triangular basis, i.e. the sums in the commutator relation (\ref{Drinfeld2}) allow to be infinite.

$(2)$~By \cite{Cramer}, the Drinfeld double Hall algebras are invariant up to derived equivalences. Thus, by Theorem \ref{tgdly}, the $2$-periodic extended derived Hall algebras are also invariant up to derived equivalences. For any two hereditary abelian categories $\A$ and $\mathcal {B}$, one conjectures that their root categories are triangulated equivalent if and only if they are derived equivalent. Hence, the $2$-periodic extended derived Hall algebras may also be called the {\em extended derived Hall algebra of the root category} $\mathcal {R}(\A)$.

$(3)$~Recently, for any hereditary abelian category $\A$ with Euler form skew symmetric, Chen, Lu and Ruan \cite{CLR2} have also defined a Hall algebra for the root category $\mathcal {R}(\A)$ by counting the triangles in the root category. In fact, if the Euler form of $\A$ is skew symmetric, i.e. $(\alpha,\beta)=0$~for any $\alpha,\beta\in K(\A)$, by Remark \ref{dyjx} we obtain that the elements $K_\alpha K_\beta^\ast$ with any $\alpha,\beta\in K(\A)$ are central in $\mathcal {D}\mathcal {H}_2(\A)$. In this case, the twisted derived Hall algebra of $\mathcal {R}(\A)$ defined in \cite{CLR2} is isomorphic to the central reduction $\mathcal {D}\mathcal {H}_2(\A)/\lr{K_\alpha K_\beta^\ast-1~|~\alpha,\beta\in K(\A)}$ of the $2$-periodic extended derived Hall algebra.
\end{remark}

Similar to \cite{Br,LP}, the {\em reduced $2$-periodic extended derived Hall algebra} $\mathcal {D}\mathcal {H}_2^{\rm red}(\A)$ of $\A$ is defined to be the quotient
of $\mathcal {D}\mathcal {H}_2(\A)$ by the ideal generated by the elements $K_\alpha K_\alpha^*-1$ for all $\alpha\in K(\A)$. We recall that the {\em reduced Drinfeld double Hall algebra} $\mathcal {D}_{\rm red}(\A)$ of $\A$ is the quotient of $\mathcal {D}(\A)$ by the ideal generated by the elements $K_\alpha \otimes 1-1\otimes K_{-\alpha}$ for all $\alpha\in K(\A)$. By Theorem \ref{tgdly}, it is easy to see that the algebras $\mathcal {D}\mathcal {H}_2^{\rm red}(\A)$ and $\mathcal {D}_{\rm red}(\A)$ are isomorphic.
Hence, we can use the reduced $2$-periodic extended derived Hall algebra to provide a global realization of the corresponding quantum enveloping algebra.

We remark that our construction has the same generality as the work by Lu-Peng \cite{LP}, i.e. it also applies to hereditary abelian categories that may not have enough projectives. By Theorem \ref{tgdly} together with \cite[Theorem 1.27]{Yan} and \cite[Theorem 4.9]{LP}, the $2$-periodic extended derived Hall algebra $\mathcal {D}\mathcal {H}_2(\A)$ is isomorphic to Bridgeland's Hall algebra (cf. \cite{Br}) and the twisted semi-derived Ringel--Hall algebra (cf. \cite{LP}) of $2$-periodic complexes associated to $\A$. Hence, taking $\A$ to be the category of coherent sheaves on a weighted projective line $\mathbb{X}$, we can use $\mathcal {D}\mathcal {H}_2^{\rm red}(\A)$ to realize the whole quantum loop algebra $\mathbf{U}_v(\mathcal {L}\mathfrak{g})$ for some $\mathfrak{g}$ with a star-shaped Dynkin diagram (cf. \cite{BurS1,BurS2,DJX}); taking $\A$ to be the category of finite-dimensional nilpotent representations of a finite quiver $Q$ with loops, we can use $\mathcal {D}\mathcal {H}_2^{\rm red}(\A)$ to realize the corresponding whole quantum Borcherds--Bozec algebra (cf. \cite{Kang, Lu}).
\section{Appendix: 1-periodic derived Hall algebras}
In this section, we apply the derived Hall numbers of $D^b(\A)$ to define a Hall algebra for the $1$-periodic derived category $\mathcal {D}_1(\A)$, and then compare it with the derived Hall algebra of $\mathcal {D}_1(\A)$ defined in \cite{XuChen}.
\subsection{$1$-periodic derived Hall algebra $\mathcal {D}\mathcal {H}_1(\A)$}
\begin{definition}\label{maindef2}
The Hall algebra $\mathcal {D}\mathcal {H}_1(\A)$, called the {\em $1$-periodic derived Hall algebra} of $\A$, is the $\mathbb{C}$-vector space with the basis $\{u_{[M]}~|~[M]\in {\rm Iso}(\A)\}$, and with the multiplication defined on basis elements by
$$u_{[A]}u_{[B]}=v^{\lr{\hat{A},\hat{B}}}\sum\limits_{[I],[M]\in{\rm Iso}(\A)}\frac{H_{I[1]\oplus A,B\oplus I[-1]}^M}{a_I}u_{[M]}.$$
\end{definition}
\begin{theorem}
The $1$-periodic derived Hall algebra $\mathcal {D}\mathcal {H}_1(\A)$ is an associative algebra.
\end{theorem}
\begin{proof}
For any $A,B,C\in\A$, we need to prove that $(u_{[A]}u_{[B]})u_{[C]}=u_{[A]}(u_{[B]}u_{[C]})$.
On the one hand, by definition,
\begin{flalign*}
&(u_{[A]}u_{[B]})u_{[C]}=\sum\limits_{[I_1],[X],[I_2],[M]}v^{\lr{\hat{A},\hat{B}}}\frac{H_{I_1[1]\oplus A,B\oplus I_1[-1]}^X}{a_{I_1}}v^{\lr{\hat{X},\hat{C}}}\frac{H_{I_2[1]\oplus X,C\oplus I_2[-1]}^M}{a_{I_2}}u_{[M]}\\
&=\sum\limits_{[I_1],[X],[I_2],[M]}v^{\lr{\hat{A},\hat{B}}}\frac{H_{I_1[1]\oplus A,B\oplus I_1[-1]}^X}{a_{I_1}}v^{\lr{\hat{A}+\hat{B}-2\hat{I}_1,\hat{C}}}\frac{H_{I_2[1]\oplus X,C\oplus I_2[-1]}^M}{a_{I_2}}u_{[M]}\\
&=v^{\lr{\hat{A},\hat{B}}+\lr{\hat{A},\hat{C}}+\lr{\hat{B},\hat{C}}}\sum\limits_{[I_1],[X],[I_2],[M]}q^{-\lr{\hat{I}_1,\hat{C}}}\frac{H_{I_1[1]\oplus A,B\oplus I_1[-1]}^X}{a_{I_1}}\frac{H_{I_2[1]\oplus X,C\oplus I_2[-1]}^M}{a_{I_2}}u_{[M]}.
\end{flalign*}

On the other hand,
\begin{flalign*}
&u_{[A]}(u_{[B]}u_{[C]})=\sum\limits_{[I_2],[X],[I_1],[M]}v^{\lr{\hat{A},\hat{X}}}\frac{H_{I_1[1]\oplus A,X\oplus I_1[-1]}^M}{a_{I_1}}v^{\lr{\hat{B},\hat{C}}}\frac{H_{I_2[1]\oplus B,C\oplus I_2[-1]}^X}{a_{I_2}}u_{[M]}\\
&=\sum\limits_{[I_1],[X],[I_2],[M]}v^{\lr{\hat{A},\hat{B}+\hat{C}-2\hat{I}_2}}\frac{H_{I_1[1]\oplus A,X\oplus I_1[-1]}^M}{a_{I_1}}v^{\lr{\hat{B},\hat{C}}}\frac{H_{I_2[1]\oplus B,C\oplus I_2[-1]}^X}{a_{I_2}}u_{[M]}\\
&=v^{\lr{\hat{A},\hat{B}}+\lr{\hat{A},\hat{C}}+\lr{\hat{B},\hat{C}}}\sum\limits_{[I_1],[X],[I_2],[M]}q^{-\lr{\hat{A},\hat{I}_2}}\frac{H_{I_1[1]\oplus A,X\oplus I_1[-1]}^M}{a_{I_1}}\frac{H_{I_2[1]\oplus B,C\oplus I_2[-1]}^X}{a_{I_2}}u_{[M]}.
\end{flalign*}
Hence, the associativity law $(u_{[A]}u_{[B]})u_{[C]}=u_{[A]}(u_{[B]}u_{[C]})$ is equivalent to the following formula.\end{proof}
\begin{lemma}
For any $A,B,C,M\in\A$, we have that
\begin{equation}\label{1jhlgs}
\begin{split}
&\sum\limits_{[I_1],[X],[I_2]}q^{-\lr{\hat{I}_1,\hat{C}}}\frac{H_{I_1[1]\oplus A,B\oplus I_1[-1]}^X}{a_{I_1}}\frac{H_{I_2[1]\oplus X,C\oplus I_2[-1]}^M}{a_{I_2}}\\&=
\sum\limits_{[I_1],[X],[I_2]}q^{-\lr{\hat{A},\hat{I}_2}}\frac{H_{I_1[1]\oplus A,X\oplus I_1[-1]}^M}{a_{I_1}}\frac{H_{I_2[1]\oplus B,C\oplus I_2[-1]}^X}{a_{I_2}}.
\end{split}\end{equation}
\end{lemma}
\begin{proof}
On the one hand,
using Corollary \ref{tuilun2.7}, we have that
\begin{flalign*}&\mbox{LHS~~of}~~(\ref{1jhlgs})
=\sum\limits_{[I_1],[X],[I_2]}q^{-\lr{\hat{I}_1,\hat{C}}}q^a\frac{a_Aa_Ba_{I_1}}{a_X}F_{A,I_1[-1],I_1[1],B}^Xq^b\frac{a_Xa_Ca_{I_2}}{a_M}F_{X,I_2[-1],I_2[1],C}^M,
\end{flalign*}
where $a=\lr{\hat{I}_1,\hat{I}_1}-\lr{\hat{A},\hat{I}_1}-\lr{\hat{I}_1,\hat{B}}$ and $b=\lr{\hat{I}_2,\hat{I}_2}-\lr{\hat{X},\hat{I}_2}-\lr{\hat{I}_2,\hat{C}}=\lr{\hat{I}_2,\hat{I}_2}-\lr{\hat{A}+\hat{B}-2\hat{I}_1,\hat{I}_2}-\lr{\hat{I}_2,\hat{C}}$.

Using Corollary \ref{tuilun1}, we obtain that
\begin{flalign*}&\mbox{LHS~~of}~~(\ref{1jhlgs})\\
&=\sum\limits_{[I_1],[X],[I_2]\atop [P_1],[P_2],[P_3],[P_4]}q^{x}\frac{a_Aa_Ba_Ca_{I_1}a_{I_2}}{a_M}F_{I_1,P_2}^AF_{P_1,I_1}^BF_{P_2,P_1}^X\frac{a_{P_1}a_{P_2}}{a_Aa_B}
F_{I_2,P_4}^XF_{P_3,I_2}^CF_{P_4,P_3}^M\frac{a_{P_3}a_{P_4}}{a_Xa_C}\end{flalign*}
\begin{flalign*}=\sum\limits_{[I_1],[X],[I_2]\atop [P_1],[P_2],[P_3],[P_4]}q^{x}\frac{a_{I_1}a_{I_2}a_{P_1}a_{P_2}a_{P_3}a_{P_4}}{a_M}F_{I_1,P_2}^AF_{P_1,I_1}^BF_{P_3,I_2}^CF_{P_4,P_3}^M(F_{P_2,P_1}^X
F_{I_2,P_4}^X\frac{1}{a_X})
\end{flalign*}
where $x=a+b-\lr{\hat{I}_1,\hat{C}}$.

Using Green's formula, we get that
\begin{flalign*}&\mbox{LHS~~of}~~(\ref{1jhlgs})
=\sum\limits_{[I_1],[I_2], [P_1],[P_2],[P_3]\atop[P_4],[T_1],[T_2],[T_3],[T_4]}q^{ x'}\frac{a_{I_1}a_{I_2}a_{P_1}a_{P_2}a_{P_3}a_{P_4}}{a_M}F_{I_1,P_2}^AF_{P_1,I_1}^BF_{P_3,I_2}^CF_{P_4,P_3}^M\\
&\quad\quad\quad\quad\quad\quad\quad\quad\quad\quad\quad\quad F_{T_1,T_2}^{P_2}F_{T_3,T_4}^{P_1}F_{T_1,T_3}^{I_2}F_{T_2,T_4}^{P_4}\frac{a_{T_1}a_{T_2}a_{T_3}a_{T_4}}{a_{P_1}a_{P_2}a_{P_4}a_{I_2}},
\end{flalign*}
where $x'=x-\lr{\hat{T}_1,\hat{T}_4}$.

Noting that $\hat{I}_2=\hat{T}_1+\hat{T}_3$ and applying the associativity formula (\ref{associativity}), we obtain that
\begin{equation}\label{1zuobian}
\begin{split}
\mbox{LHS~~of}~~(\ref{1jhlgs})
=\sum\limits_{[I_1],[P_3],[T_1]\atop [T_2],[T_3],[T_4]}q^{x'}\frac{a_{I_1}a_{P_3}a_{T_1}a_{T_2}a_{T_3}a_{T_4}}{a_M}F_{I_1,T_1,T_2}^AF_{T_3,T_4,I_1}^BF_{P_3,T_1,T_3}^CF_{T_2,T_4,P_3}^M.
\end{split}\end{equation}

On the other hand, using Corollary \ref{tuilun2.7}, we have that
\begin{flalign*}\mbox{RHS~~of}~~(\ref{1jhlgs})=\sum\limits_{[I_1],[X],[I_2]}q^{-\lr{\hat{A},\hat{I}_2}}q^{b'}\frac{a_Aa_Xa_{I_1}}{a_M}F_{A,I_1[-1],I_1[1],X}^Mq^{a'}\frac{a_Ba_Ca_{I_2}}{a_X}F_{B,I_2[-1],I_2[1],C}^X,
\end{flalign*}
where $a'=\lr{\hat{I}_2,\hat{I}_2}-\lr{\hat{B},\hat{I}_2}-\lr{\hat{I}_2,\hat{C}}$ and $b'=\lr{\hat{I}_1,\hat{I}_1}-\lr{\hat{A},\hat{I}_1}-\lr{\hat{I}_1,\hat{B}+\hat{C}-2\hat{I}_2}$.

Using Corollary \ref{tuilun1}, we obtain that
\begin{flalign*}&\mbox{RHS~~of}~~(\ref{1jhlgs})\\
&=\sum\limits_{[I_1],[X],[I_2]\atop [Q_1],[Q_2],[Q_3],[Q_4]}q^{y}\frac{a_Aa_Ba_Ca_{I_1}a_{I_2}}{a_M}F_{I_2,Q_2}^BF_{Q_1,I_2}^CF_{Q_2,Q_1}^X\frac{a_{Q_1}a_{Q_2}}{a_Ba_C}
F_{I_1,Q_4}^AF_{Q_3,I_1}^XF_{Q_4,Q_3}^M\frac{a_{Q_3}a_{Q_4}}{a_Aa_X}\\&
=\sum\limits_{[I_1],[X],[I_2]\atop [Q_1],[Q_2],[Q_3],[Q_4]}q^{y}\frac{a_{I_1}a_{I_2}a_{Q_1}a_{Q_2}a_{Q_3}a_{Q_4}}{a_M}F_{I_1,Q_4}^AF_{I_2,Q_2}^BF_{Q_1,I_2}^CF_{Q_4,Q_3}^M(F_{Q_2,Q_1}^X
F_{Q_3,I_1}^X\frac{1}{a_X})
\end{flalign*}
where $y=a'+b'-\lr{\hat{A},\hat{I}_2}$.

Using Green's formula, we get that
\begin{flalign*}&\mbox{RHS~~of}~~(\ref{1jhlgs})=\sum\limits_{[I_1],[I_2], [Q_1],[Q_2],[Q_3]\atop[Q_4],[T'_1],[T'_2],[T'_3],[T'_4]}q^{y'}\frac{a_{I_1}a_{I_2}a_{Q_1}a_{Q_2}a_{Q_3}a_{Q_4}}{a_M}F_{I_1,Q_4}^AF_{I_2,Q_2}^BF_{Q_1,I_2}^CF_{Q_4,Q_3}^M\\
&\quad\quad\quad\quad\quad\quad\quad\quad\quad\quad\quad\quad F_{T'_1,T'_2}^{Q_2}F_{T'_3,T'_4}^{Q_1}F_{T'_1,T'_3}^{Q_3}F_{T'_2,T'_4}^{I_1}\frac{a_{T'_1}a_{T'_2}a_{T'_3}a_{T'_4}}{a_{Q_1}a_{Q_2}a_{Q_3}a_{I_1}},
\end{flalign*}
where $y'=y-\lr{\hat{T}'_1,\hat{T}'_4}$.

Noting that $\hat{I}_1=\hat{T}'_2+\hat{T}'_4$ and applying the associativity formula (\ref{associativity}), we obtain that
\begin{equation}\label{1youbian}
\begin{split}
\mbox{RHS~~of}~~(\ref{1jhlgs})=\sum\limits_{[I_2],[Q_4],[T'_1]\atop [T'_2],[T'_3],[T'_4]}q^{y'}\frac{a_{I_2}a_{Q_4}a_{T'_1}a_{T'_2}a_{T'_3}a_{T'_4}}{a_M}F_{T'_2,T'_4,Q_4}^AF_{I_2,T'_1,T'_2}^BF_{T'_3,T'_4,I_2}^CF_{Q_4,T'_1,T'_3}^M.
\end{split}\end{equation}

Replacing the notations $I_1,P_3,T_1,T_2,T_3,T_4$ in (\ref{1zuobian}) by $T'_2,T'_3,T'_4,Q_4,I_2,T'_1$, respectively,
we get that all terms in (\ref{1zuobian}) are the same as those in (\ref{1youbian}) except the exponents of $q$.

Now let us compare the exponents of $q$ in (\ref{1zuobian}) and (\ref{1youbian}). Replacing the notations in (\ref{1zuobian}) as above, we have that
\begin{equation}\label{lcommon1}
\begin{split}
x'=&\lr{\hat{T}'_2,\hat{T}'_2}-\lr{\hat{A},\hat{T}'_2}-\lr{\hat{T}'_2,\hat{B}}+\lr{\hat{T}'_4+\hat{I}_2,\hat{T}'_4+\hat{I}_2}-\lr{\hat{A}+\hat{B}-2\hat{T}'_2,\hat{T}'_4+\hat{I}_2}\\
&-\lr{\hat{T}'_4+\hat{I}_2,\hat{C}}-\lr{\hat{T}'_2,\hat{C}}-\lr{\hat{T}'_4,\hat{T}'_1}
\end{split}\end{equation}
and
\begin{equation}\label{rcommon1}
\begin{split}
y'=&\lr{\hat{I}_2,\hat{I}_2}-\lr{\hat{B},\hat{I}_2}-\lr{\hat{I}_2,\hat{C}}+\lr{\hat{T}'_2+\hat{T}'_4,\hat{T}'_2+\hat{T}'_4}-\lr{\hat{A},\hat{T}'_2+\hat{T}'_4}\\
&-\lr{\hat{T}'_2+\hat{T}'_4,\hat{B}+\hat{C}-2\hat{I}_2}-\lr{\hat{A},\hat{I}_2}-\lr{\hat{T}'_1,\hat{T}'_4}.
\end{split}\end{equation}
By direct calculations, we get that
\begin{equation*}
x'=\lr{\hat{I}_2,\hat{T}'_4}-\lr{\hat{B},\hat{T}'_4}+\lr{\hat{T}'_2,\hat{T}'_4}-\lr{\hat{T}'_4,\hat{T}'_1}+\text{Common~terms}~(\spadesuit)
\end{equation*}
and
\begin{equation*}
y'=\lr{\hat{T}'_4,\hat{T}'_2}-\lr{\hat{T}'_4,\hat{B}}+\lr{\hat{T}'_4,\hat{I}_2}-\lr{\hat{T}'_1,\hat{T}'_4}+\text{Common~terms}~(\spadesuit),
\end{equation*}
where $(\spadesuit)$ denotes the common terms of $(\ref{lcommon1})$ and $(\ref{rcommon1})$.

Noting that $\hat{B}-\hat{I}_2-\hat{T}'_2=\hat{T}'_1$, we can obtain that $x'=y'$.
Therefore, we complete the proof.
\end{proof}

\subsection{$1$-periodic derived Hall algebra $\mathcal {D}\mathcal {H}_{\mathbb{Z}_1}(\A)$}
The derived Hall algebras of odd periodic triangulated categories satisfying some finiteness conditions have been defined in \cite{XuChen}. In particular, the derived Hall algebra of the $1$-periodic derived category $\mathcal {D}_1(\A)$ has been defined. For any $A,B\in\mathcal {D}_1(\A)$, set $$[A,B]:=\frac{1}{|\Hom_{\mathcal {D}_1(\A)}(A,B)|}.$$

Note that $\Iso(\mathcal {D}_1(\A))=\Iso(\A)$. So for any $A\in\A$, in order to distinguish, we write $\tilde{a}_A=|\Aut_{\mathcal {D}_1(\A)}(A)|$
and ${a}_A=|\Aut_{\A}(A)|$, and we have that $\tilde{a}_A=a_A\cdot|\Ext_{\A}^1(A,A)|$ (cf. \cite[Lemma 3.1]{CLR}). According to \cite[Corollary 2.7]{XuChen}, for any $A,B,M\in\A$, we have that
$$\frac{|\Hom_{\mathcal {D}_1(\A)}(B,M)_A|}{\tilde{a}_B}\sqrt{\frac{[B,M]}{[B,B]}}=\frac{|\Hom_{\mathcal {D}_1(\A)}(M,A)_B|}{\tilde{a}_A}\sqrt{\frac{[M,A]}{[A,A]}}=:G_{A,B}^M.$$
By \cite[Proposition 3.5]{CLR},
\begin{equation}\label{1-Hallnumber}
G_{A,B}^M=\sum\limits_{[L],[I],[N]\in{\rm Iso}(\A)}v^{\lr{\hat{I},\hat{N}}+\lr{\hat{I},\hat{I}}+\lr{\hat{L},\hat{I}}-\lr{\hat{L},\hat{N}}}\frac{a_La_Ia_N}{a_Aa_B}F_{N,L}^MF_{I,N}^AF_{L,I}^B.
\end{equation}
The {\em $1$-periodic derived Hall algebra} $\mathcal {D}\mathcal {H}_{\mathbb{Z}_1}(\A)$ of $\A$ is the $\mathbb{C}$-vector space with the basis $\{\mu_{[X]}~|~[X]\in {\rm Iso}(\A)\}$, and with the multiplication defined by
\begin{equation}\label{odhallm}
\mu_{[A]} \mu_{[B]}=\sum\limits_{[M]\in {\rm Iso}(\A)}G_{A,B}^M\mu_{[M]}.\end{equation}
\begin{theorem}
There exists an algebra isomorphism
$$\Phi:\mathcal {D}\mathcal {H}_{1}(\A)\longrightarrow\mathcal {D}\mathcal {H}_{\mathbb{Z}_1}(\A), u_{[M]}\mapsto v^{-\lr{\hat{M},\hat{M}}}\cdot a_M\cdot \mu_{[M]}.$$
\end{theorem}
\begin{proof}
Since $\Phi$ sends a basis of $\mathcal {D}\mathcal {H}_{1}(\A)$ to a basis of $\mathcal {D}\mathcal {H}_{\mathbb{Z}_1}(\A)$, clearly, $\Phi$ is a bijection. We only need to prove that $\Phi(u_{[A]}u_{[B]})=\Phi(u_{[A]})\Phi(u_{[A]})$ for any $A,B\in\A$.

On the one hand,
\begin{flalign*}
&\Phi(u_{[A]}u_{[B]})=v^{\lr{\hat{A},\hat{B}}}\sum\limits_{[I],[M]\in{\rm Iso}(\A)}\frac{H_{I[1]\oplus A, B\oplus I[-1]}^M}{a_I}\Phi(u_{[M]})\\
&=v^{\lr{\hat{A},\hat{B}}}\sum\limits_{[I],[M]\in{\rm Iso}(\A)}q^{-\lr{\hat{A},\hat{I}}-\lr{\hat{I},\hat{B}}}{a_Aa_Ba_I}F_{A\oplus I[1],I[-1]\oplus B}^Mv^{-\lr{\hat{M},\hat{M}}}\mu_{[M]}\quad (\text{by~Corollary~\ref{tuilun2.7}}).
\end{flalign*}
Noting that \begin{flalign*}F_{A\oplus I[1],I[-1]\oplus B}^M&=q^{\lr{\hat{I},\hat{I}}}\sum\limits_{[L],[N]\in{\rm Iso}(\A)}F_{N,L}^MF_{I[1],B}^LF_{A,I[-1]}^N \quad (\text{by~Proposition~\ref{gongshi1})}\\
&=q^{\lr{\hat{I},\hat{I}}}\sum\limits_{[L],[N]\in{\rm Iso}(\A)}F_{N,L}^M\frac{a_La_N}{a_Aa_B}F_{I,N}^AF_{L,I}^B\quad (\text{by~Lemma~\ref{diyiy})}\end{flalign*}
and $\hat{M}=\hat{A}+\hat{B}-2\hat{I}$, we obtain that
\begin{flalign*}&\Phi(u_{[A]}u_{[B]})=\\
&v^{-\lr{\hat{A},\hat{A}}-\lr{\hat{B},\hat{B}}-\lr{\hat{B},\hat{A}}}\sum\limits_{[L],[I],[N],[M]\in{\rm Iso}(\A)}q^{\lr{\hat{I},\hat{A}}+\lr{\hat{B},\hat{I}}-\lr{\hat{I},\hat{I}}}a_La_Ia_NF_{N,L}^MF_{I,N}^AF_{L,I}^B\mu_{[M]}.
\end{flalign*}

On the other hand, by the equation (\ref{1-Hallnumber}), we have that
\begin{equation}\label{last}
\begin{split}
&\Phi(u_{[A]})\Phi(u_{[B]})=v^{-\lr{\hat{A},\hat{A}}}a_A\mu_{[A]}v^{-\lr{\hat{B},\hat{B}}}a_B\mu_{[B]}\\
&=v^{-\lr{\hat{A},\hat{A}}-\lr{\hat{B},\hat{B}}}\sum\limits_{[L],[I],[N],[M]\in{\rm Iso}(\A)}v^{\lr{\hat{I},\hat{N}}+\lr{\hat{I},\hat{I}}+\lr{\hat{L},\hat{I}}-\lr{\hat{L},\hat{N}}}{a_La_Ia_N}F_{N,L}^MF_{I,N}^AF_{L,I}^B\mu_{[M]}.\end{split}\end{equation}
Noting that $\hat{N}=\hat{A}-\hat{I}$ and $\hat{L}=\hat{B}-\hat{I}$ in (\ref{last}), we can easily complete the proof.
\end{proof}

\section*{Acknowledgments}
The authors would like to thank the anonymous referee for the careful reading, helpful comments and suggestions.
This work was partially supported by the National Natural Science Foundation of China (Grant No. 12271257).

\end{document}